%% file: two_minimal_spheres.tex
\documentclass[11pt]{amsart}
\usepackage{amssymb}
\usepackage{amsfonts}
\usepackage{amsmath}
\usepackage{graphicx}
\usepackage{xcolor}
\usepackage{amsmath}
\usepackage{mathrsfs}
\usepackage{stmaryrd}
\usepackage{epsfig,color}
\usepackage{blindtext}
\usepackage{enumerate}
\usepackage{hyperref}
\usepackage{url}
\usepackage{bbm}
\usepackage{nicefrac,mathtools}
\usepackage{bm}   
\usepackage{tabularx}
\DeclareGraphicsExtensions{.pdf,.jpeg,.png}
\usepackage{epstopdf}
\usepackage{cancel} 
\usepackage[normalem]{ulem} 
\usepackage{verbatim} 
\usepackage{scalerel}
\usepackage{enumitem}
\usepackage{tikz-cd}

\usepackage{color}
\usepackage[msc-links, lite]{amsrefs}
\usepackage{geometry}
\usepackage{stackengine,scalerel}

\newtheorem{theorem}{Theorem}[section]

\newtheorem{problem}{Problem}
\newtheorem*{theoremA}{Theorem A}

\newtheorem{proposition}[theorem]{Proposition}
\newtheorem{lemma}[theorem]{Lemma}

\newtheorem*{claim}{Claim}

\theoremstyle{definition}
\newtheorem{definition}[theorem]{Definition}
\theoremstyle{remark}
\newtheorem{remark}[theorem]{Remark}

\newtheorem*{rmk}{Remark}

\numberwithin{equation}{section}

\newcommand{\mf}{\mathbf}
\newcommand{\mb}{\mathbb}
\newcommand{\mc}{\mathcal}
\newcommand{\ms}{\mathscr}
\newcommand{\mk}{\mathfrak}
\newcommand{\mr}{\mathrm}
\newcommand{\oli}{\overline}

\newcommand{\wti}{\widetilde}
\newcommand{\res}{\scaleobj{1.75}{\llcorner}}

\newcommand{\C}{\mathcal C}

\newcommand{\rom}[1]{\expandafter\romannumeral #1}
\newcommand{\Rom}[1]{\uppercase\expandafter{\romannumeral #1}}

\DeclareMathOperator{\rel}{\mathrm{rel}}

\DeclareMathOperator{\VC}{\mc V\C}

\title[Existence of two minimal 2-spheres]{Existence of two embedded minimal spheres in $S^3$ with an arbitrary metric}

\author{Zhichao Wang}
\address{Shanghai Center for Mathematical Science, 2005 Songhu Road, Fudan University, Shanghai, 200438, China}
\email{zhichao@fudan.edu.cn}

\author{Xin Zhou}
\address{Department of Mathematics, 531 Malott Hall, Cornell University, Ithaca, NY 14853, USA}
\email{xinzhou@cornell.edu}

\begin{document}
\begin{abstract}
    We prove that $S^3$ endowed with an arbitrary Riemannian metric $g$ admits at least two embedded minimal spheres. The proof is based on an iterative scheme of relative min-max constructions.
\end{abstract}

\maketitle

\section{Introduction}

In 1982, S. T. Yau posed the following problem \cite{Yau82}: 
\begin{problem}[\cite{Yau82}*{Problem 89}]\label{prob}
Prove that there are four distinct embedded minimal spheres in any (Riemannian) manifold diffeomorphic to $S^3$ (the 3-dimensional sphere).  
\end{problem}

Let $(S^3, g)$ denote an arbitrary Riemannian manifold diffeomorphic to $S^3$. In this paper, we establish the following result (see Theorem \ref{thm:two spheres}, where the theorem is restated and proved):

\begin{theoremA}
    There exist at least two distinct embedded minimal spheres in $(S^3, g)$. 
\end{theoremA}

Shortly after Yau posed this problem, Simon--Smith \cite{Smith82} proved the existence of at least one embedded minimal sphere in any $(S^3, g)$. Their proof utilized a variant of the min-max theory for minimal hypersurfaces pioneered by Almgren \cites{Alm62,Alm65} and Pitts \cite{Pi}; (see also \cite{SS} and \cite{Colding-DeLellis03}). 
It is also worth noting that branched immersed minimal spheres were obtained earlier by Sacks--Uhlenbeck \cite{Sacks-Uhlenbeck81} via a min-max theory for harmonic maps; (see also \cite{Colding-Minicozzi08b}).

Although there are naturally four distinct families of embedded spheres that can be plugged into the Simon--Smith min-max theory to construct embedded minimal spheres, the issue of integer multiplicity poses a significant obstacle to establishing the existence of a second minimal sphere. For metrics with positive Ricci curvature, White \cite{Whi91} employed degree-theoretic methods to prove the existence of at least two embedded minimal spheres, and at least four when the metric is sufficiently close to the round metric. More recently, Haslhofer--Ketover \cite{HK19} proved the existence of at least two embedded minimal spheres for bumpy metrics by demonstrating that the second width has multiplicity one.

Building on these developments, the authors \cite{wangzc2023existenceFour} very recently proved a multiplicity-one theorem for the Simon--Smith min-max theory, resolving Yau's conjecture for $(S^3, g)$ when $g$ has positive Ricci curvature or when $g$ is bumpy. We remark that this work utilized the regularity results established by Sarnataro--Stryker \cite{SS23} and the authors \cite{Wang-Zhou-C11-2023}. In the present paper, we extend these advancements to prove the existence of at least two embedded minimal spheres for all Riemannian 3-spheres. Notably, this marks the first time that more than one embedded minimal sphere has been constructed for general metrics, advancing the general problem for the first time since Simon--Smith \cite{Smith82} found the initial minimal sphere over 40 years ago.

\medskip

More recently, the Simon--Smith min-max theory has seen extensive study, not only concerning the existence of minimal surfaces with controlled topology but also in the context of broader problems in geometric topology. Building on the multiplicity-one theorem established by the authors \cite{wangzc2023existenceFour}, many results (e.g., \citelist{\cite{chu-Li-tori}\cite{chu-all-genus}\cite{Li-W-X-lens}\cite{Chu-Li-Wang-genus}\cite{Chu-Li-Wang-enumerative}\cite{li-Wang-tori}\cite{WWZ-S4}}) have addressed the existence of minimal surfaces with controlled genus in Riemannian manifolds (primarily in dimension 3) with positive Ricci curvature. Furthermore, Ketover-Liokumovich \cite{Liokomovich-Ketover23} applied our multiplicity-one theorem in their proof of Smale conjecture for lens spaces. Additionally, Ambrozio--Marques--Neves \cite{AMN-rigidity} recently proved that surface Zoll metrics on the three-dimensional sphere are characterized by the equality of their spherical area widths. To successfully employ multiplicity-one theorems, many of these recent advancements rely on the assumption that the metric has positive Ricci curvature. We hope that the methods developed in this paper can be applied to investigate similar problems without requiring such Ricci curvature assumptions.

\medskip

Finally, we notice the tremendous development and success in the original Almgren--Pitts min-max theory in the past decade, for example \cites{MN14, LMN16, MN17, IMN17, MNS17, Song18, Zhou19, MN18}, leading to the resolution of the Willmore conjecture by Marques--Neves \cite{MN14} and Yau's abundance conjecture for minimal surfaces \cite{Yau82}*{Problem 88} by Marques--Neves \cite{MN17} and Song \cite{Song18}.

\subsection{Difficulties and ideas of the proof}

Recall that the Simon--Smith min-max theory always yields at least one embedded minimal sphere, denoted by $\Sigma_0$. 
When $\Sigma_0$ is a stable minimal sphere with \emph{contracting neighborhoods} (\cite{Song18}*{Page 883}) --- that is, a neighborhood foliated by surfaces with their mean curvature vectors pointing towards $\Sigma_0$, --- we can cut $S^3$ along $\Sigma_0$ and construct embedded minimal spheres on both sides by applying an argument from \cite{Song18}. In what follows, we only consider the case where $\Sigma_0$ has an \emph{expanding neighborhood} (that is, a small neighborhood foliated by constant mean curvature surfaces whose mean curvature vectors point away from $\Sigma_0$).

We first recall several classical methods for constructing additional minimal spheres and explain why they fail for a general Riemannian metric.

Recall that there exist four different classes of sweepouts (with min-max widths $\mf L_1 \leq \mf L_2 \leq \mf L_3 \leq \mf L_4$) on the space of embedded spheres in $S^3$ \cite{wangzc2023existenceFour}*{\S 8.3}. Moreover, by Lusternik--Schnirelmann theory, it suffices to consider the case where
\[ 
    \mf L_1 < \mf L_2 < \mf L_3 < \mf L_4. 
\]
However, the possibility of higher multiplicity in the Simon--Smith min-max theory becomes the major barrier in constructing the second minimal sphere. In prior literature, there are three main methods utilized to construct additional solutions:

\begin{itemize}
    \item \emph{Degree theory}. 
    White \cite{Whi91} defined a degree for the space of embedded minimal spheres in $S^3$ with positive Ricci curvature. Nevertheless, this degree may not be well-defined for a general $S^3$ due to the lack of compactness of the space of embedded minimal spheres. 

    \smallskip

    \item \emph{Multiplicity-one theorems}. In our previous work \cite{wangzc2023existenceFour}, we proved a multiplicity-one theorem for unstable components produced by the Simon--Smith min-max theory, which implies the existence of at least four minimal spheres in a Riemannian 3-sphere with positive Ricci curvature. We also note that Haslhofer--Ketover \cite{HK19} used catenoid estimates (proved by Ketover--Marques--Neves \cite{KMN16}) to establish multiplicity one for the second min-max width if $S^3$ has positive Ricci curvature. However, in another previous work \cite{Wang-Zhou22}, we constructed metrics on $S^3$ such that $\mf L_2$ can only be realized by $\Sigma_0$ with multiplicity two. In other words, it is impossible to obtain the second minimal sphere by expecting multiplicity one to hold for all general metrics.

    \smallskip

    \item \emph{Symmetry of $S^3$}. There are also constructions \cite{BP-nonplanar} of embedded minimal spheres on ellipsoids, but heavily relying on the symmetries of the ellipsoids.
\end{itemize}

\medskip
We proceed by contradiction. Suppose that $\Sigma_0$ is the unique embedded minimal sphere in $(S^3, g)$. Then, by our previous work \cite{wangzc2023existenceFour}, $\Sigma_0$ has an expanding neighborhood. Moreover, by Haslhofer-Ketover \cite{HK19}*{Theorem 3.1}, there exists a foliation $\{\Upsilon_t\}_{t \in [0,1]}$ of $S^3$ such that $\Upsilon(\frac{1}{2}) = \Sigma_0$. Let $\mc{Z}_2(S^3; \mb{Z}_2)$ denote the space of 2-dimensional mod 2 integral cycles in $(S^3, g)$, and let $\Delta_k$ denote the simplex $\{ (x_1, \cdots, x_k) \mid 0 \leq x_1 \leq x_2 \leq \cdots \leq x_k \}$. 
Recall that the map
\[ 
    \Psi_k: \Delta_k \to \mc{Z}_2(S^3; \mb{Z}_2), \quad (x_1, \cdots, x_k) \mapsto \sum_{i=1}^k \llbracket \Upsilon(x_i) \rrbracket 
\]
induces a $k$-sweepout (in the sense of Gromov \cite{Gro88}; see also \cite{MN17}) from $\mb{RP}^k$ to $\mc{Z}_2(S^3; \mb{Z}_2)$. By the Smale conjecture (proved by Hatcher \cite{Hat83}; see also Bamler-Kleiner \cite{Bam-Klei-smale}), the space of embedded spheres in $S^3$ is homotopic to $\mb{RP}^3$, which implies that we cannot construct nontrivial sweepouts from $\mb{RP}^k$ to the space of embedded spheres when $k > 4$. Instead, we will construct a family of spheres $\{ \Phi(x) \}_{x \in \Delta_k}$ to approximate $\{ \Psi_k(x) \}_{x \in \Delta_k}$ in the following sense:

\begin{itemize}
    \item $\Phi(x)$ has area strictly less than $k \mc{H}^2(\Sigma_0) + \eta$, where $\eta > 0$ is a given small constant.
    \item $\Phi(x)$ has area strictly less than $(k-1) \mc{H}^2(\Sigma_0) + \eta$ for all $x \in \partial \Delta_k$.
    \item $\Phi(x)$ has area strictly less than $(k-1) \mc{H}^2(\Sigma_0) - \eta$ if $x_i = x_{i+1}$ for some $1 \leq i \leq k-1$.
    \item $\Phi(0, x')$ and $\Phi(x', 1)$ bound ``almost" complementary Caccioppoli sets for all $x' \in \Delta_{k-1}$.
\end{itemize}

We then consider relative Simon--Smith min-max constructions for $\{ \Phi(x) \}_{x \in \Delta_k}$. Let $W_k$ denote the min-max value. When $k=1$, by the Simon--Smith theory, $W_1$ is realized by an integral varifold supported on embedded minimal spheres. By the uniqueness of $\Sigma_0$, we have $W_1 = \mc{H}^2(\Sigma_0)$. Inductively, we can prove that $W_k > W_{k-1}$ provided that $W_{k-1} = (k-1) \mc{H}^2(\Sigma_0)$ by a Lusternik–Schnirelmann type argument. By the uniqueness of $\Sigma_0$ again, $W_k\geq k\mc H^2(\Sigma_0)$. By the area upper bounds for $\Phi(x)$, we conclude that $W_k=k\mc H^2(\Sigma_0)$.
However, for $k$ large enough, we can show that $\sup_{x\in \Delta_k}\mc{H}^2(\Phi(x)) < k \mc{H}^2(\Sigma_0)$. This leads to a contradiction, thereby proving the existence of the second minimal sphere.

To explain this upper bound, we must provide more details regarding the construction of $\{ \Phi(x) \}_{x \in \Delta_k}$. For $x \in \Delta_k$ away from a tubular neighborhood of $\partial \Delta_k$, we can simply define $\Phi(x)$ by deforming $\cup_{i=1}^k \Upsilon(x_i)$ as follows: we replace two small disks on $\Upsilon(x_i)$ and $\Upsilon(x_{i+1})$ and insert a neck connecting them. This produces a Lipschitz sphere consisting of $(2k-1)$ pieces of smooth surfaces with boundary, which we will eventually smooth out. The primary difficulty lies in how to effectively extend this map continuously to all of $\Delta_k$. The key idea is to define the radius $\xi_i$ of the necks based on the distance between $x_i$ and $x_{i+1}$ (modulated by a parameter $\Lambda\gg 1$). For instance, we enforce $\xi_{i-1} + \xi_i \to 1^-$ if $x \in \Delta_k \setminus \partial \Delta_k$ converges to $\hat{x} = (\hat{x}_1, \cdots, \hat{x}_k) \in \partial \Delta_k$ with $\hat{x}_i = \hat{x}_{i+1}$. The parameter $\Lambda$ controls the rate at which $\xi_{i-1} + \xi_i \to 1^-$ as $x \to \hat{x}$. We refer the reader to Section \ref{sec: construct family of spheres} for the delicate details of this construction.

For a fixed parameter $\Lambda_0$ and a sufficiently large integer $k_0$, any $x = (x_1, \cdots, x_{k_0}) \in \Delta_{k_0}$ must contain at least one adjacent pair $(x_i, x_{i+1})$ whose difference is arbitrarily small. This ensures the existence of at least one neck with a large radius and a small height, resulting in a uniformly small area. However, the two disks removed to attach this neck have areas strictly bounded from below. Consequently, the total area of $ \Phi(x)$ is strictly less than $k_0 \mc{H}^2(\Sigma_0)$.

To carry out the Lusternik-Schnirelmann argument, we require the sum of the Caccioppoli sets bounded by $\Phi(0,x')$ and $\Phi(x',1)$ being sufficiently close to $\llbracket S^3\rrbracket$ as mod 2 currents. For a given $k_0$, this can be achieved by taking $\Lambda$ sufficiently large; however, doing so may violate the area bound $\sup\mc{H}^2(\Phi(x)) < k_0 \mc{H}^2(\Sigma_0)$. To simultaneously preserve this strict area upper bound and the requirement for $\Phi(0,x')$ and $\Phi(x',1)$, we will interpolate between $\Phi_{\Lambda}$ (for large $\Lambda > \Lambda_0$) near the boundary $\partial \Delta_{k_0}$, and $\Phi_{\Lambda_0}$ for $x$ away from $\partial\Delta_{k_0}$.

\subsection{Organization} 
This paper is organized as follows. In Section \ref{sec:pre}, we set up notations and review essential lemmas from geometric measure theory. We also state two relative Simon--Smith min-max theorems that will be applied later. In Section \ref{sec: construct family of spheres}, we construct a family of spheres and present fundamental estimates for the area of the connecting necks. In Section \ref{sec:area upper bound}, we establish upper bounds for the area of this family as well as related properties. In Section \ref{sec:smooth family}, we smooth out the family and show that the resulting smooth family retains the desired area upper bounds. In Section \ref{sec:apply min-max}, we determine the min-max values of this smooth family of spheres, a process that relies heavily on a crucial topological lemma (Lemma \ref{lem:top}). Finally, in Appendix \ref{app:approximation}, we prove that our Lipschitz family of spheres can be appropriately approximated by a smooth family.

\subsection*{Acknowledgments} We would like to thank Xingzhe Li for valuable suggestions. Z.W. is supported by Tianyuan Mathematics Frontier Key Special Program (NSFC Grant No. 12526203). X.Z. is supported by NSF grants DMS-2506717 and a grant from the Simons Foundation. 
 
\section{Preliminaries}\label{sec:pre}

\subsection{Notations}
We will assume that $(S^3, g)$ is isometrically embedded in some $\mb R^L$.
\begin{itemize}
    
    \item $\mc C(S^3)$ denotes the space of sets $\Omega\subset S^3$ with finite perimeter (Caccioppoli set), \cite{Si}*{\S 14}. We sometime denote $\llbracket \Omega \rrbracket$ by the associated 3-dimensional mod-2 integral current.

    \item $\mc Z_2(S^3; \mb Z_2)$ denote the space of 2-dimensional mod-2 integral cycles in $(S^3, g)$, \cite{Fed}*{\S 4.1}.
    
    \item $\mc V(S^3)$ denotes the space of $2$-varifolds in $S^3$, \cite{Si}*{\S 38}.    
    
    \item The \textit{$\VC$-space} defined in \cite{wangzc2023existenceFour}*{Definition 1.3}, denoted by $\VC(S^2)$, is the space of all pairs $(V, \Omega)\in \mc V(S^3)\times\mc C(S^3)$ such that there is a sequence $\{\Omega_k\}\subset \mc C(S^3)$ with $|\partial\Omega_k|\to V$ in $\mc V(S^3)$ and $\Omega_k\to \Omega$ in $\mc C(S^3)$.
    
    \item Given two pairs $(V, \Omega)$ and $(V', \Omega')$ in $\VC(S^3)$, the \textit{$\ms F$-distance between them} defined in \cite{wangzc2023existenceFour}*{Definition 1.3} is
    \begin{equation*}
        \ms F\big( (V, \Omega), (V', \Omega') \big):= \mf F(V, V') + \mc F(\Omega, \Omega'),
    \end{equation*}
    where $\mf F$ and $\mc F$ are respectively the varifold $\mf F$-metric in $\mc V(S^3)$ and the flat metric in $\mc C(S^3)$. 

    \item Given a connected closed surface $\Sigma_0$, we denote $[\Sigma_0]$ or $\llbracket \Sigma_0 \rrbracket$ respectively by the associated $2$-varifold or mod-$2$ integral current. 
    
    \item Denote by (see \cite{wangzc2023existenceFour}*{(2.1)}) 
    \[ \ms E(\Sigma_0):=\{(\Sigma, \Omega) \in \VC(M) \big| \Sigma \text{ is a separating embedding of $\Sigma_0$ which bounds } \Omega\}, \]
    endowed with the oriented smooth topology in the usual sense, that is, $(\Sigma_j, \Omega_j)$ converges to $(\Sigma_\infty, \Omega_\infty)$ if $\Sigma_j$ converges in the smooth topology to $\Sigma_\infty$ and $\Omega_j$ converges to $\Omega_\infty$ in $\C(M)$. 

    \item Denote by
    \[ \oli{\ms E}(\Sigma_0):=\left\{(\Sigma, \Omega) \in \VC(M) \left|\,
    \begin{aligned}
    &\Sigma \text{ is a point or a piecewise-smooth, separating }\\
    &\text{ embedding of $\Sigma_0$ which bounds } \Omega
    \end{aligned}\right\} \right.
    \]
    Note that if $(p, \Omega)\in \oli{\ms E}$ where $p$ is a point, then $\Omega = \emptyset$ or $M$. In the following, we may sometime abuse notation and write
    \[ \mc H^2(\Sigma, \Omega) = \mc H^2(\Sigma). \]

    \item When $\Sigma_0$ is the 2-sphere $\mb S^2$, we simply write $\ms E(\mb S^2)$ as $\ms E$ and $\overline{\ms E(\mb S^2)}$ as $\overline{\ms E}$. At times, we abuse notation by using $\ms E$ to denote the space of smooth embedded 2-spheres. 

    \item Given a map $\Phi: X\to \ms E \text{ or } \oli{\ms E}$ for some parameter space $X$, by writing 
    \[ \Phi(x) = \big(\Phi^1(x),\Phi^2(x)\big), \quad \text{for each $x\in X$}, \]
    we mean that $\Phi^1(x)$ is the associated surface and $\Phi^2(x)$ is the associated Caccioppoli set.
\end{itemize}

\subsection{A varifold neighborhood lemma}

Assume that $\Sigma_0$ is only smooth embedded minimal 2-sphere in $(S^3, g)$. For $\tau>0$, denote by $\mc U_{\tau}(k|\Sigma_0|)$ the set of integral varifolds $V$, such that $\mf F(V, k|\Sigma_0|)<\tau$. Letting $k\geq 0$ be an integer, we have the following.

\begin{proposition}\label{prop:no flipping near a varifold}
There exist $\tau_0>0$ and $\theta_0>0$ depending only on $(S^3, g)$ and $\Sigma_0$, but independent of $k$, such that if $\Phi: [0,1] \to \mc C(S^3)$ is a continuous family satisfying 
    \begin{equation}\label{eq:small neighborhood}
        |\partial \Phi(x)| \in \mc U_{\tau_0}(k|\Sigma_0|), \quad \forall x\in [0,1],
    \end{equation}
    then 
    \[
     \mc F\big(\Phi(0)+\Phi(1),\llbracket S^3\rrbracket\big) >\theta_0.
    \]
\end{proposition}
\begin{proof}
Let $B_{2r_0}(x_0)$ be a geodesic ball disjoint from $\Sigma_0$. By the local isoperimetric theorem \cite{Ritore2023}*{Lemma 1.46},  
there exists a constant $A_0 > 0$, such that
\[
     \mf M(\partial T) \geq A_0 \quad \text{for all $T\in \mc C(B_{2r_0}(x_0))$ with $\frac{1}{3} \mc H^3(B_{2r_0}(x_0)) < \mf M(T) < \frac{2}{3} \mc H^3(B_{2r_0}(x_0))$}.
\]
We can choose a non-negative Lipschitz function $f\in \text{Lip}(S^3, g)$ such that 
    \[ f\equiv c_0 \in (0, 1] \text{ on } B_{r_0}(x_0), \quad f \equiv 0 \text{ on } S^3\setminus B_{2r_0}(x_0), \quad \text{and} \quad \text{Lip}(f)\leq 1 \text{ on } S^3. \]

Now let $\theta_0=\frac{1}{3} \mc H^3(B_{r_0}(x_0))$ and $\tau_0=c_0A_0$. Suppose, for the sake of contradiction, that $\Phi:[0,1]\to \mc C(S^3)$ satisfies 
   \[
     \mc F(\Phi(0)+\Phi(1),\llbracket S^3\rrbracket) < \theta_0, \quad \text{but } |\partial\Phi(0)|,|\partial\Phi(1)|\in \mc U_{\tau_0} (k|\Sigma_0|).
    \]
Then,
 \begin{align*}
      c_0A_0 = \tau_0 \geq  \mf F(\Phi(0), k|\Sigma_0|) & \geq 
      \left| \int f\,\mr d|\partial \Phi(0)| - \int f\, \mr d (k|\Sigma_0|) \right | \\
       & = \int f \,\mr d|\partial \Phi(0)|\geq c_0\mf M(\partial \Phi(0) \res B_{r_0}(x_0)).
    \end{align*}
Here the second $\geq$ follows from the definition of $\mf F$-norm \cite{Pi}*{\S 2.1(19)}.     
Hence we obtain
\[
\mf M(\partial \Phi(0) \res B_{r_0}(x_0))\leq A_0,
\]
which implies that 
\begin{equation}\label{eq:Phi0 small or large}
    \mf M(\Phi(0)\res B_{r_0}(x_0)) \leq \frac{1}{3}\mc H^3(B_{r_0}(x_0)) \quad \text{or} \quad \mf M(\Phi(0)\res B_{r_0}(x_0)) \geq \frac{2}{3}\mc H^3(B_{r_0}(x_0)).
\end{equation}
Similarly,
\begin{equation}\label{eq:Phi1 small or large}
    \mf M(\Phi(1)\res B_{r_0}(x_0)) \leq \frac{1}{3}\mc H^3(B_{r_0}(x_0)) \quad \text{or} \quad \mf M(\Phi(1)\res B_{r_0}(x_0)) \geq \frac{2}{3}\mc H^3(B_{r_0}(x_0)).
\end{equation}
Without loss of generality, we may assume that
\begin{equation}\label{eq:Phi0 < Phi1}
    \mf M(\Phi(0)\res B_{r_0}(x_0))\leq \mf M(\Phi(1)\res B_{r_0}(x_0)).
\end{equation}
Since 
\[
\mc F(\Phi(0)+\Phi(1),\llbracket S^3\rrbracket) < \frac{1}{3}\mc H^3(B_{r_0}(x_0)),
\]
it follows that
\[
      \mc F\big((\Phi(0)+\Phi(1))\res B_{r_0}(x_0),\llbracket B_{r_0}(x_0)\rrbracket\big) < \frac{1}{3}\mc H^3(B_{r_0}(x_0)),
\]
which leads to 
\[
   \frac{2}{3}\mc H^3(B_{r_0}(x_0))\leq   \mf M(\Phi(0)\res B_{r_0}(x_0)) +  \mf M(\Phi(1)\res B_{r_0}(x_0)) \leq \frac{4}{3}\mc H^3(B_{r_0}(x_0)).
\]
Combining this with \eqref{eq:Phi0 small or large}, \eqref{eq:Phi1 small or large}, and \eqref{eq:Phi0 < Phi1}, we conclude that 
\[
     \mf M(\Phi(0)\res B_{r_0}(x_0)) \leq \frac{1}{3}\mc H^3(B_{r_0}(x_0)) \quad \text{and} \quad \mf M(\Phi(1)\res B_{r_0}(x_0)) \geq \frac{2}{3}\mc H^3(B_{r_0}(x_0)).
\]  
By continuity of $\Phi$, there exists $t_0\in(0,1)$ such that 
\[
    \mf M(\Phi(t_0)\res B_{r_0}(x_0)) = \frac{1}{2}\mc H^3(B_{r_0}(x_0)). 
\]
Then we have 
\begin{align*}
       \mf F(\Phi(t_0), k|\Sigma_0|) & \geq 
      \left| \int f\,\mr d|\partial \Phi(t_0)| - \int f\, \mr d (k|\Sigma_0|) \right | \\
       & = \int f \,\mr d|\partial \Phi(t_0)|\geq c_0\mf M(\partial \Phi(t_0) \res B_{r_0}(x_0))>\tau_0.
    \end{align*}
This contradicts the assumption \eqref{eq:small neighborhood}, completing the proof of Proposition \ref{prop:no flipping near a varifold}.
\end{proof}

\subsection{Simon-Smith min-max theory}\label{SS:SS min-max}

We present a relative version of the Simon-Smith min-max theory \cite{Smith82}, adopting the notations in \cite{wangzc2023existenceFour}.
Fix a positive integer $m\in \mb N$ and denote $I(m)=[0,1]^m$. Given any positive integer $n\in \mb N$, let $I(m,n)$ be the cubic complex given by 
$I(m, n)=I(1, n)\otimes \cdots I(1, n)$ ($n$-times), where $I(1, n)$ denotes the complex on $I=[0,1]$ whose $1$-cells and $0$-cells are, respectively, $[1, 3^{-n}], [3^{-n}, 2\cdot 3^{-n}], \cdots, [1-3^{-n}, 1]$ and $[0], [3^{-n}], \cdots, [1-3^{-n}], [1]$. 
See \cite{Zhou19}*{Appendix A} for a summary of these notions. 

Let $Z\subset I(m,n)$ be a cubical subcomplex and $\Phi: Z\to \ms E$ be a continuous map. Denote by $[\Phi; \rel Z]$ the collection of continuous maps $\Psi: X \to \ms E$ such that $\partial X=Z$ (where $X\subset I(m,n')$ is a cubical complex for some positive integer $n'\geq n$) and $\Psi|_{\partial X = Z}=\Phi$. 
Denote by 
\[
    \mf L([\Phi; \rel Z]):= \inf_{\Psi\in [\Phi;\rel Z]} \sup_{x\in \mr{dom} \Psi} \mc H^2(\Psi(x)).
\]

\begin{theorem}[Relative Simon-Smith Min-Max]\label{thm:relative min-max theorem}
    With all notions as above, suppose
    \begin{equation}\label{eq:width nontrivial1}
         \mf L([\Phi; \rel Z]) > \sup_{x\in Z}\mc H^2\big(\Phi(x)\big).
    \end{equation}
    Then there exists a stationary integral varifold $V$ supported on embedded minimal spheres such that
    \[ \mf L([\Phi; \rel Z]) = \|V\|(M).\] 
\end{theorem}
\begin{proof}
Let $\{\Psi_i=(\Psi_i^1,\Psi_i^2):X_i\to \ms E\}$ be a sequence of maps such that $\Psi_i\in [\Phi;\rel Z]$ and 
\[
    \sup_{x\in X_i} \mc H^2(\Psi_i^1(x)) \searrow \mf L([\Phi;\rel Z]).
\]
Let $\Pi_i$ be the collection of maps $ \Psi'=(\Psi'^1,\Psi'^2):X_i\to \ms E$ that is homotopic to $\Psi_i$ with $\Psi'|_{Z} = \Psi_i|_{Z} = \Phi$.  By the relative Simon-Smith min-max theorem (cf. \cite{wangzc2023existenceFour}*{Theorem 7.3}), there exists a stationary integral varifold $V_i$ supported on embedded minimal spheres with 
\[
  \|V_i\|(M) = \mf L(\Pi_i):=\inf_{\Psi'\in \Pi_i} \sup_{x\in X_i} \mc H^2(\Psi'^1(x)).
\]
Note that we have $\mf L(\Pi_i)\searrow \mf L([\Phi;\rel Z])$ since $\Pi_i\subset [\Phi;\rel Z]$. Then $V_i$ converges to a stationary integral varifold $V$ supported on embedded minimal spheres by compactness of embedded minimal 2-spheres with bounded area; see \citelist{\cite{CS}\cite{Whi87}}. This proves Theorem \ref{thm:relative min-max theorem}.
\end{proof}

\smallskip

For our later purpose, we need to consider min-max problems associated with a sequence of homotopy classes whose min-max values converge. To state the setups, fix a cubic subcomplex $Z \subset I(m, n)$ and a sequence of continuous maps $\{\Phi_i: Z\to \ms E\}_{i\in \mb N}$. Assume that
\begin{equation}\label{eq:width nontrivial2}
    \mf L:= \liminf_{i\to \infty}\mf L([\Phi_i; \rel Z]) > \sup_{i\in \mb N}\left\{\sup_{x\in Z}\mc H^2(\Phi_i(x))\right\}.
\end{equation}
Let $\{X_i\}_{i\in \mb N}$ be a sequence of cubic complexes with $X_i\subset I(m, n_i)$ ($n_i\geq n$) and $\partial X_i=Z$, and 
$\{\Psi_i:X_i\to \ms E\}_{i\in \mb N}$ be a sequence of continuous maps with $\Psi_i|_Z = \Phi_i$, such that (up to a subsequence which we do not re-label)
\[ \mf L(\Psi_i): = \sup_{x\in X_i}\mc H^2\big(\Psi_i(x)\big) \searrow \mf L , \text{ when } i\to\infty. \]
Such a sequence $\{\Psi_i:X_i\to \ms E\}$ is called a \textit{minimizing sequence} associated with the sequence of boundary data $\{\Phi_i: Z\to \ms E\}$. 
A subsequence $\{\Psi_{i_j}(x_j): x_j\in X_{i_j}\}_{j\in \mb N}$ is called a \textit{min-max (sub)sequence} if 
\[ \mc H^2\big( \Psi_{i_j}(x_j) \big) \to \mf L , \text{ when } j \to\infty. \]
The \textit{critical set} of a minimizing sequence $\{\Psi_i=(\Psi_i^1,\Psi_i^2)\}$ is defined by
\[ \mf C(\{\Psi_i\})=\left\{V\in\mc V(S^3)\left|\,
    \begin{aligned}   
        & \exists \text{ a min-max subsequence }\{\Psi^1_{i_j}(x_j)\} \text{ such}\\
        & \text{that } \mf F(\Psi^1_{i_j}(x_j), V) \to  0 \text{ as } j\to\infty
    \end{aligned}\right\}\right..
\]

\begin{theorem}\label{thm:min-max thm for sequences}
With all notions as above, suppose that $\{\Psi_i: X_i\to \ms E\}$ is a minimizing sequence associated with boundary data $\{\Phi_i: Z\to \ms E\}$ satisfying \eqref{eq:width nontrivial2}. 
Then there exists a min-max sequence $\{\Psi^1_{i_j}(x_j): x_j\in X_{i_j}\}_{j\in \mb N}$  
which converges to a stationary integral varifold $V$ supported on embedded minimal spheres such that 
\[
    \mf L= \|V\|(M). 
\]
 \end{theorem}
\begin{proof} 
Let $\Pi_i$ be the homotopy class generated by $\Psi_i:X_i\to \ms E$ relative to the boundary value $\Phi_i: Z\to \ms E$. 
By definition, $\Psi_i\in \Pi_i\subset [\Phi_i; \rel Z]$.
Define
\[ 
    \mf L( \Pi_i): = \inf_{ \Psi'\in \Pi_i} \sup_{x\in X_i} \mc H^2(\Psi'(x)).
\]
Then we have
\[
    \mf L([\Phi_i;\rel Z]) \leq \mf L(\Pi_i) \leq \mf L(\Psi_i),
\]
which immediately implies that $\mf L(\Pi_i)\to \mf L$ as  $i\to \infty$.
Noting \eqref{eq:width nontrivial2}, by the same pull-tight process as in \cite{wangzc2023existenceFour}*{Theorem 2.7}, there exists another minimizing sequence $\{\Psi_i^*\in \Pi_i\}_{i\in \mb N}$, such that $\mf C(\{\Psi_i^*\})\subset \mf C(\{\Psi_i\})$ and every element $V\in \mf C(\{\Psi_i^*\})$ is stationary. Then by a second tightening process \cite{wangzc2023existenceFour}*{Lemma 3.10} using a combinatorial argument originally due to Pitts \cite{Pi} (noting that $X_i$ embeds to $I(m, n_i)$ for a common $m$), we can find a min-max sequence $\{\Psi^*_{i_j}(x_j)\}$ converging to a stationary varifold $V_0\in \mc V (S^3)$, which is \textit{almost minimizing in small annuli}; (we refer the definition to the statement in \cite{wangzc2023existenceFour}*{Theorem 3.8}). By the regularity of almost minimizing varifold \cites{Smith82, Colding-DeLellis03} (see also \cite{wangzc2023existenceFour}), we know that $V_0$ is integral and supported on embedded minimal spheres. As $\mf C(\{\Psi_i^*\})\subset \mf C(\{\Psi_i\})$, we know that $V_0$ is also the limit of some min-max sequence $\{\Psi^1_{i_j}(x_j)\}$. This completes the proof of Theorem \ref{thm:min-max thm for sequences}.
\end{proof}

\section{Construction of multi-parameter family of spheres}\label{sec: construct family of spheres}

By the Simon-Smith theory \cite{Smith82}, $(S^3, g)$ contains at least one embedded minimal 2-sphere $\Sigma_0$. Furthermore, if $\Sigma_0$ is the only embedded minimal 2-sphere in $(S^3, g)$, then by \cite{wangzc2023existenceFour}*{Theorem 8.9} we know that $\Sigma_0$ is weakly stable; that is, the first eigenvalue $\lambda_1(L_{\Sigma_0})$ of the associated Jacobi operator $L_{\Sigma_0}$ vanishes, $\lambda_1(L_{\Sigma_0})=0$. We may also assume that $\Sigma_0$ has an expanding neighborhood (\cite{Song18}*{Page 883}) on both sides. Hence, by \cite{HK19}*{Theorem 3.1} (see also \cite{Deaibes2-foliations}), $\Sigma_0$ can be extended to a mean convex foliation on each side of $\Sigma_0$.

\medskip

Let $\Upsilon: [0,1]\to \ms E$ be the mean convex foliation of $S^3$ by 2-spheres with $\Upsilon(\frac{1}{2}) = \Sigma_0$. Suppose that $\Upsilon(0)=p$ and $\Upsilon(1)=q$, where $p, q$ are two points in $S^3$. Let $\gamma_\pm:[0,1]\to S^3$ be two distinct curves joining $p$ and $q$ such that 
\begin{itemize}
    \item $\gamma_\pm(t)\in \Upsilon(t)$;
    \item $\gamma_+$ intersects $\gamma_-$ only at $p$ and $q$;
    \item $\gamma_\pm$ are transverse to $\Upsilon(t)$ for all $t\in(0,1)$.
\end{itemize}
Denote by 
\[ \text{$\Omega(t)$ or $\Omega_t$ the connected component of $S^3\setminus \Upsilon(t)$ which contains $\Upsilon(0)$.} \]
Then $\Omega(0) = \emptyset$ and $\Omega(1) = S^3$.

\medskip

Observe that there exists a deformation retract  $\mc R:S^3\setminus \gamma_+ \times [0,1] \to S^3\setminus \gamma_+$ satisfying the following properties:
\begin{enumerate}
    \item $\mc R(\Upsilon(t)\setminus \gamma_+(t), s ) \subset \Upsilon(t)\setminus \gamma_+(t)$ for all $t,s \in[0,1]$;
    \item $\mc R(\Upsilon(t)\setminus \gamma_+(t), s ) \subset \mc R(\Upsilon(t)\setminus \gamma_+(t), s' )$ for all $t\in[0,1]$ and $0\leq s'\leq s\leq 1$;
    \item $\mc R(\cdot, 0)=\mr{id}$ on $S^3\setminus \gamma_+$;
    \item $\mc R(\Upsilon(t)\setminus \gamma_+(t), 1) =\gamma_-(t)$;
    \item for each $s, t\in(0,1)$, $\mc R(\Upsilon(t)\setminus \gamma_+(t), s)$ is a disk-type surface whose boundary curve has length uniformly bounded from above.
\end{enumerate}
Denote by 
\[ \text{$\mc D(s)= \mc R(S^3\setminus \gamma_+,s)$ \quad for all $s\in [0,1]$.} \]
Then $\mc D(0)= S^3\setminus \gamma_+(0,1)$ and $\mc D(1)= \gamma_-([0,1])$.

\medskip

We introduce the standard $k$-dimensional parameter simplex, denoted by 
\[
    \Delta_k:=\{(x_1,\cdots, x_k)\in[0,1]^k: x_1\leq x_2\leq \cdots\leq x_k\}. 
\]
When $k$ is fixed, we sometimes write $x_0 = 0$ and $x_{k+1} = 1$. 

For any $x = (x_1, x_2, \dots, x_k) \in \Delta_k$, we define the adjacent coordinate gaps as 
\[ d_i = x_{i+1} - x_i, \quad \text{ for } i = 1, \dots, k-1.\]
Given $\Lambda>0$, we construct the map 
\[ \xi^\Lambda = (\xi_1^\Lambda, \xi_2^\Lambda, \dots, \xi_{k-1}^\Lambda) : \Delta_k \to [0,1]^{k-1} \]
by recursively defining the components. Setting the base case $\xi_0^\Lambda := 0$, we define:
\begin{equation}\label{eq:xilambda}
    \xi_i^\Lambda = e^{-\Lambda d_i}(1 - \xi_{i-1}^\Lambda), \quad \text{for } i = 1, 2, \dots, k-1.
\end{equation}
We also define $\xi_k^\Lambda := 0$. We always omit the index $\Lambda$ when there is no ambiguity.

\begin{proposition}\label{prop:properties of xi}
The above construction satisfies the following properties:
\begin{enumerate}[label=(\roman*)]
    \item The map $\xi : \mb R_{>0}\times \Delta_k \to [0,1]^{k-1}$ is continuous, and $\xi$ maps $\mb R_{>0}\times (\Delta_k\setminus \partial \Delta_k)$ into $(0, 1)^{k-1}$.
    
    \item  For any index $i$, the fact that $d_i \ge 0$ ensures that $e^{-\Lambda d_i} \leq 1$; applying this to \eqref{eq:xilambda} yields $\xi_i \le 1 \cdot (1 - \xi_{i-1})$, that is, $\xi_{i-1} + \xi_i \le 1$.

    \item  Given $\epsilon>0$, suppose $d_i > -\frac{\ln \epsilon}{\Lambda}$; since $\xi_{i-1} \ge 0$ and hence $(1 - \xi_{i-1}) \le 1$, we have the upper bound $\xi_i \le e^{-\Lambda d_i}  \leq  e^{-\Lambda d} = \epsilon$.
\end{enumerate}
\end{proposition}

\subsection{Construction of family of spheres}

We now use the continuous map $\xi$ to construct a continuous family of piecewise-smooth embedded spheres. 
\begin{definition}\label{def:Phi_Lambda}
Given $\Lambda>0$ and a positive integer $k$, we define 
\[ \Phi_{\Lambda}=(\Phi_{\Lambda}^1,\Phi_{\Lambda}^2):\Delta_k \to \oli{\ms E} \]
as follows; see Figure \ref{fig:k-family}. 
\begin{enumerate}
    \item $\Phi^1_{\Lambda}(x_1, x_2, \dots, x_k)$ is defined by the union of the interior annuli (or disks):
    \[ \Upsilon(x_i) \cap \mathcal{D}(\xi_{i}) \setminus \mathcal{D}(1-\xi_{i-1})  \text{ for all odd $i$ with $1\leq i\leq k$}, \]
    \[ \Upsilon(x_i) \cap \mathcal{D}(\xi_{i-1}) \setminus \mathcal{D}(1-\xi_i)  \text{ for all even $i$ with $1\leq i\leq k$}, \]
    along with the connecting boundaries (the ``necks'') as follows:
    \[ \partial \mathcal{D}(\xi_i)\cap (\Omega_{x_{i+1}}\setminus\oli{\Omega_{x_i}}) =  \bigcup_{x_i \le x \le x_{i+1}} \partial \left( \Upsilon(x) \cap \mathcal{D}(\xi_i) \right), \quad \text{for odd } i \in [1, k-1],\]
    \[ \partial \mathcal{D}(1-\xi_i)\cap (\Omega_{x_{i+1}}\setminus\oli{\Omega_{x_i}}) = \bigcup_{x_i \le x \le x_{i+1}} \partial \left( \Upsilon(x) \cap \mathcal{D}(1-\xi_i) \right), \quad \text{for even } i\in[1, k-1]. \]
    \begin{rmk}
        Proposition \ref{prop:properties of xi} (ii) ensures that $\mathcal{D}(\xi_i)\supset \mathcal{D}(1-\xi_{i-1})$ and $\mathcal{D}(\xi_{i-1}) \supset \mathcal{D}(1-\xi_i)$, so $\Phi^1_\Lambda$ is well-defined.
    \end{rmk}

    \item $\Phi_{\Lambda}^2:\Delta_k\to \mc C(S^3)$ is the union of the following open sets:
    \[  (\Omega_{x_{i+1}}\setminus\oli{\Omega_{x_i}})\setminus \oli{\mc D(\xi_i)} = \bigcup_{x_i < x < x_{i+1}}  \left( \Upsilon(x) \setminus \oli{\mathcal{D}(\xi_i)} \right)  \quad \text{for odd } i \in [1, k], \]
    \[ (\Omega_{x_{i+1}}\setminus\oli{\Omega_{x_i}}) \setminus \oli{\mathcal{D}(1-\xi_i) }= \bigcup_{x_i < x < x_{i+1}}  \left( \Upsilon(x) \setminus \oli{\mathcal{D}(1-\xi_i) }\right)   \quad \text{for even } i \in [0, k]. \]
    \begin{rmk}
        One can easily verify that $\Phi_{\Lambda}^2 (1,\cdots,1) = \llbracket S^3\rrbracket$ and 
        \[ \partial \Phi_\Lambda^2 (x_1,\cdots,x_k) =\llbracket \Phi_\Lambda^1(x_1,\cdots,x_k)\rrbracket \]
        in the sense of mod 2 integral currents.
    \end{rmk}
\end{enumerate}
\end{definition}
\begin{figure}[ht]
	\begin{center}
		\def\svgwidth{0.6\columnwidth}
		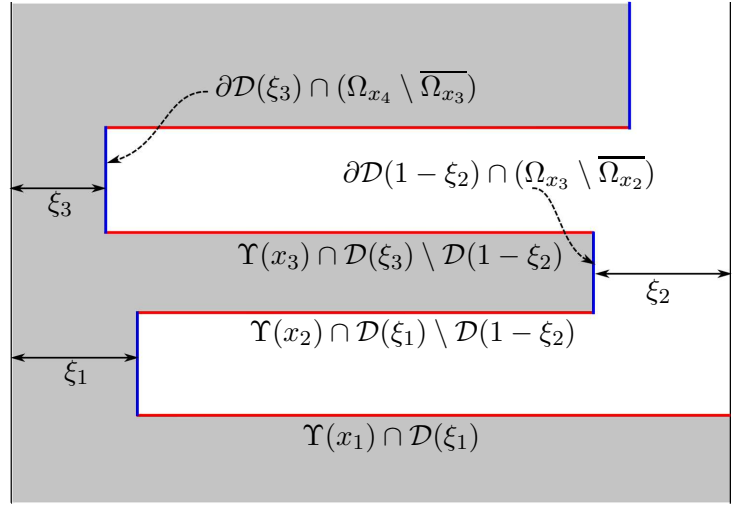
		\caption{Multiple family of spheres.}
        This is an example of $\Phi_\Lambda(x)$ for $x_1<x_2<x_3<\cdots$. The left vertical line is $\gamma_+$ and the right vertical line is $\gamma_-$. From the bottom to the top, the $i$-th red line stands for $\Phi^1_\Lambda(x)\cap \Upsilon(x_i)$; the $i$-th vertical blue one is the $i$-th neck whose radius is $\xi_i$. The gray part is $\Phi_\Lambda^2(x)$.
		\label{fig:k-family}
	\end{center}
\end{figure}

By Proposition \ref{prop:properties of xi}, we readily have the following continuity properties.
\begin{proposition}\label{prop:2k-1 pieces}\label{lem:continuity of Phi_Lambda}
Assume that $k>1$.
\begin{enumerate}
\item For $x\in \Delta_k\setminus \partial \Delta_k$ and $\Lambda>0$, $\Phi_\Lambda^1(x)$ consists of two smoothly embedded disks and $(2k-3)$ smoothly embedded annuli ($(k-2)$ interior annuli and $(k-1)$ necks), and these surfaces vary continuously with respect to $x$ in the $C^\infty$ topology. 

\item For $x=(0,\cdots,0,x_1,\cdots,x_m,1,\cdots,1)\in \Delta_k$ with $0<x_1<\cdots<x_m<1$ ($m\leq k$), $\Phi_\Lambda^1$ converges to $\sum_{i=1}^m[\Upsilon(x_i)]$ in the sense of varifolds as $\Lambda\to\infty$.  

\item The map $(\Lambda, x)\in \mb R_{>0}\times \Delta_k \mapsto \Phi_\Lambda(x)\in \oli{\ms E}$ is continuous in the $\ms F$-topology.
\end{enumerate}
\end{proposition}

We will then prove that as $\Lambda\to\infty$, $\Phi_\Lambda^2$ uniformly converges to a continuous map in the flat topology.
\begin{lemma}\label{lem:flat convergence of Phi}
Given a positive integer $k$ and a sequence $(x^\Lambda_1,\cdots, x^\Lambda_k)(\in \Delta_k)\to (x_1,\cdots, x_k)$ as $\Lambda \to \infty$, we have 
\[
     \lim_{\Lambda\to\infty} \Phi_\Lambda^2(x^\Lambda_1,\cdots, x^\Lambda_k) = \sum_{i=1}^k \llbracket \Omega(x_i) \rrbracket + (k-1)\llbracket S^3\rrbracket .
\]
\end{lemma}
\begin{proof}
As above, let $x_0=x_{k+1}=0$. Suppose that $\{x_{m_1}, x_{m_2},\cdots\}$ are all of the distinct coordinates of $x$. Then $S^3$ can be decomposed into 
\[
    S^3 = \bigcup \left(\Omega(x_{m_{i+1}})\setminus \Omega(x_{m_i})\right). 
\]
It suffices to prove the desired equality when restricted to $\Omega(x_{m_{i+1}})\setminus \Omega(x_{m_i})$ for all $i\geq 1$.

Consider $1\leq m\leq k$ such that $x_{m+1}-x_m>0$.  
Then 
\[
 (x_{m+1}^\Lambda-x_m^\Lambda) \to (x_{m+1}-x_m)>0, \text{ as $\Lambda \to \infty$}.
\]
By the definition \eqref{eq:xilambda} of $\xi^\Lambda$, we obtain that 
\begin{equation}\label{eq:xi_m goes to 0}
    \xi_m^\Lambda:= e^{-\Lambda |x_{m+1}^\Lambda-x_m^\Lambda|} (1-\xi_{m-1}^\Lambda) \to 0, \text{ as $\Lambda\to\infty$}.
\end{equation}
Suppose that $m$ is odd. Then by Definition \ref{def:Phi_Lambda} (2), we have
\[
    \Phi_\Lambda^2 (x_1,\cdots,x_k) \res \left( \Omega(x_{m+1}^\Lambda)\setminus \oli{\Omega({x_m^\Lambda})} \right) = 
    \left(\Omega(x_{m+1}^\Lambda)\setminus \oli{\Omega({x_m^\Lambda})} \right)\setminus \oli{\mc D(\xi_m^\Lambda)},
\]
which converges to $\emptyset$ in the flat topology by \eqref{eq:xi_m goes to 0}. 
Hence we obtain
\[
 \Big(\lim_{\Lambda\to\infty} \Phi_\Lambda^2(x_1^\Lambda,\cdots, x_k^\Lambda)\Big) \res \big(\Omega({x_{m+1})}\setminus \oli{\Omega({x_m})}\big) = \lim_{\Lambda\to\infty} \left(\Omega({x_{m+1}^\Lambda})\setminus \oli{\Omega({x_m^\Lambda})}\right)\setminus \oli{\mc D(\xi_m^\Lambda)}  = \emptyset .
\]
Moreover, since $m$ is odd, we have
\begin{align*}
&\ \ \ \ \left(\sum_{i=1}^k \llbracket \Omega(x_i) \rrbracket + (k-1)\llbracket S^3\rrbracket\right) \res \big(\Omega({x_{m+1}})\setminus \oli{\Omega({x_m})}\big) \\
&= (k-m) \llbracket \Omega({x_{m+1}})\setminus \oli{\Omega({x_m})}\rrbracket + (k-1)\llbracket \Omega({x_{m+1}})\setminus \oli{\Omega({x_m})}\rrbracket =0.   
\end{align*}
This gives the desired equality for odd $m$. 

Similarly, when $m$ is even, we have that
\[
    \Phi_\Lambda^2 (x_1,\cdots,x_k) \res \left(\Omega(x_{m+1}^\Lambda)\setminus \oli{\Omega({x_m^\Lambda})} \right) =
    \left(\Omega(x_{m+1}^\Lambda)\setminus \oli{\Omega({x_m^\Lambda})} \right)\setminus \oli{\mc D(1-\xi_m^\Lambda)},
\]
which converges to $\llbracket \Omega({x_{m+1}})\setminus \oli{\Omega({x_m})}\rrbracket$ in the flat topology by \eqref{eq:xi_m goes to 0}. 
Hence we obtain
\[
 \Big(\lim_{\Lambda\to\infty} \Phi_\Lambda^2(x_1^\Lambda,\cdots, x_k^\Lambda)\Big) \res \big(\Omega({x_{m+1})}\setminus \oli{\Omega({x_m})}\big) = \llbracket \Omega({x_{m+1}})\setminus \oli{\Omega({x_m})}\rrbracket .
\]
Since $m$ is even, we also have that 
\begin{align*}
& \ \ \ \ \left(\sum_{i=1}^k \llbracket \Omega(x_i) \rrbracket + (k-1)\llbracket S^3\rrbracket\right) \res \big(\Omega({x_{m+1}})\setminus \oli{\Omega({x_m})}\big) \\
& = (k-m) \llbracket \Omega({x_{m+1}})\setminus \oli{\Omega({x_m})}\rrbracket + (k-1)\llbracket \Omega({x_{m+1}})\setminus \oli{\Omega({x_m})}\rrbracket = \llbracket \Omega({x_{m+1}})\setminus \oli{\Omega({x_m})}\rrbracket.   
\end{align*}
This completes the proof Lemma \ref{lem:flat convergence of Phi}.
\end{proof}

\subsection{Area of small necks}
To consider the area of $\Phi_{\Lambda}^1$, we introduce the following notions.

For any $0\leq a\leq b\leq 1$ and $\sigma\in(0,1)$, denote by
\[
    \mc C_+ (a,b,\sigma) = \mc H^2 \left( \partial \mc D(\sigma) \cap [\Omega_b \setminus \oli{\Omega_a}]\right),
\]
and
\[
    \mc C_- (a,b,\sigma) = \mc H^2 \left(\partial \mc D(1-\sigma) \cap [\Omega_b \setminus \oli{\Omega_a}]\right).
\]
For simplicity, denote by
\begin{equation}\label{eq:C_i}
\mk C_i := 
\begin{cases} 
    \mc C_+(x_i,x_{i+1},\xi_i) & \text{ for odd } 1\leq i\leq k-1, \\[0.5em] 
    \mc C_-(x_i,x_{i+1},\xi_i) & \text{ for even } 1\leq i\leq k-1.
\end{cases}    
\end{equation}
Denote by
\begin{equation}\label{eq:A_i}
   \mk A_i:=
    \begin{cases}
        \mc H^2( \Upsilon(x_i) \cap \mathcal{D}(\xi_{i}) \setminus \mathcal{D}(1-\xi_{i-1})) &\text{ for odd } 1\leq i\leq k, \\[0.5em]
        \mc H^2( \Upsilon(x_i) \cap \mathcal{D}(\xi_{i-1}) \setminus \mathcal{D}(1-\xi_{i})) &\text{ for even } 1\leq i\leq k.
    \end{cases} 
\end{equation}
Then we have that
\[
    \mc H^2(\Phi_{\Lambda}^1) = \sum_{i=1}^{k-1} \mk C_i+ \sum_{i=1}^k \mk A_i.
\]

For all $t\in[0,1]$ and $\sigma\in[0,1]$, we also define
\begin{equation}\label{eq:E+ easy bound}
    \mc E_+(t, \sigma) := \mc H^2(\Sigma_0)- \mc H^2(\Upsilon(t)\cap \mc D(\sigma)),
\end{equation}
\begin{equation}\label{eq:E- easy bound}
    \mc E_-(t, \sigma):=\mc H^2(\Sigma_0)-\mc H^2(\Upsilon(t)\setminus \mc D(1-\sigma)).
\end{equation}

We now proceed to bound the area $\mc C_{\pm}(s, t, \sigma)$ of the connecting necks using $\mc E_\pm(x,\sigma)$. Recall that $\mc E_\pm(x,\sigma)$ measures the difference of area between the maximal slice $\Upsilon(\frac{1}{2})$ and the slice $\Upsilon(x)$ with a single disk removed. We will first demonstrate that this difference is strictly bounded from below by a positive constant whenever $\sigma$ is bounded away from zero. Furthermore, because necks with small heights possess a uniformly small area, we can conclude that the area of these short necks is strictly bounded from above by $\mc E_\pm(t,\sigma)$ provided $\sigma$ is bounded away from zero.

\begin{lemma}\label{lem:choose v}
Given $\sigma_0>0$, there exists $h_0>0$ such that 
for all $0\leq s\leq t\leq 1$ with $|t-s|\leq h_0$, $\sigma\in[\sigma_0,1]$,
\[
    \mc C_\pm(s,t,\sigma) < \frac{1}{2}\mc E_\pm(x,\sigma), \quad \text{ for all $x\in (0,1)$}.
\]
\end{lemma}
\begin{proof}
Note that 
\[
    \mc H^2(\Upsilon(t))\to 0 \text{ as $t\to 0^+$ or $t\to 1^-$}.
\]
Hence there exists $\alpha\in(0,1)$ such that 
\[
    \mc H^2(\Upsilon(t)) < \frac{1}{2}\mc H^2(\Sigma_0), \text{ for all $t\in[0,\alpha]\cup[1-\alpha,1]$}.
\] 
It follows that 
\begin{equation}\label{eq:mc E lower bound at ends}
    \mc E_\pm(t,\sigma) \geq \mc H^2(\Sigma_0) - \mc H^2(\Upsilon(t))\geq  \frac{1}{2}\mc H^2(\Sigma_0), \text{ for all $t\in[0,\alpha]\cup[1-\alpha,1]$}.
\end{equation}
Observe that $t\mapsto \mc H^2(\Upsilon(t)\setminus \mc D(\sigma_0))$ is a continuous function defined on $[0,1]$. Moreover, for $t\in[\alpha,1-\alpha]$,
\[  
    \mc H^2(\Upsilon(t)\setminus \mc D(\sigma_0))>0.
\]
Hence there exists $0<\epsilon_0<\frac{1}{2}\mc H^2(\Sigma_0)$ such that
\[  
\mc H^2(\Upsilon(t)\setminus \mc D(\sigma_0))>\epsilon_0, \text{ for all $t\in[\alpha,1-\alpha]$},
\]
which yields
\[
    \mc H^2(\Upsilon(t)\setminus \mc D(\sigma)) >\epsilon_0, \text{ for all $t\in[\alpha,1-\alpha]$ and $\sigma\in[ \sigma_0,1]$}.
\]
Thus we have that 
\[
\begin{aligned}
    \mc E_+(t,\sigma) & = \mc H^2(\Sigma_0) - \big(\mc H^2(\Upsilon(t)) - \mc H^2(\Upsilon(t)\setminus \mc D(\sigma)) \big) \\
    & \geq \mc H^2(\Upsilon(t)\setminus \mc D(\sigma)) >\epsilon_0, \text{ for all $t\in[\alpha,1-\alpha]$ and $\sigma\in [\sigma_0,1]$}.
\end{aligned}   
\]
This together with \eqref{eq:mc E lower bound at ends} gives that 
\begin{equation}\label{eq:E+ lower bd for all t}
    \mc E_+(t,\sigma) > \epsilon_0, \text{ for all $t\in[0,1]$ and $\sigma\in[\sigma_0,1]$}.
\end{equation}
Similarly, by choosing $\epsilon_0$ small enough, we also have that 
\begin{equation}\label{eq:E- lower bd for all t}
    \mc E_-(t,\sigma) > \epsilon_0, \text{ for all $t\in[0,1]$ and $\sigma\in[\sigma_0,1]$}.
\end{equation}

Observe that 
\[
    \sup \{\mc C_\pm(s,s+h,\sigma);s\in(0,1-h),\sigma\in(0,1)\}\to 0, \text{ as $h\to 0$}.
\]
Hence we can take $h_0>0$ small enough such that 
\[
     \mc C_\pm(s,t,\sigma)\leq \frac{1}{2}\epsilon_0, \text{ for all $0\leq s\leq t\leq 1$ with $(t-s)\leq h_0$},
\]
which together with \eqref{eq:E+ lower bd for all t} and \eqref{eq:E- lower bd for all t} implies that
\[
    \mc C_\pm(s,t,\sigma) \leq \frac{1}{2}\epsilon_0<\frac{1}{2}\mc E_\pm(x,\sigma), \text{ for all $0\leq s\leq t\leq 1$ with $(t-s)\leq h_0$}.
\]
This finishes the proof of Lemma \ref{lem:choose v}.
\end{proof}

By our construction of the piecewise smooth spheres, necks that are not short must have correspondingly small radii, resulting in a uniformly small area. Hence, we can successfully bound the neck area by $\mc E_\pm$ for all admissible cases.

\begin{lemma}\label{lem:small necks}
Given $\eta>0$, there exist $\Lambda_0>1$ such that for all $0<s\leq t\leq 1$ and $0\leq \sigma\leq e^{-\Lambda_0|t-s|}$,
\begin{equation}\label{eq:small necks}
    \mc C_\pm (s,t,\sigma)\leq \max\{\frac{1}{2}\mc E_\pm(x,\sigma), \eta\cdot (t-s)\}, \text{ for all $s\leq x\leq t$}.
\end{equation}
The equality holds if and only if $s=t=\frac{1}{2}$ and $\sigma=0$.
\end{lemma}
\begin{proof}

By \eqref{eq:mc E lower bound at ends}, there exists $\alpha\in (0,1)$ such that for all $ t\in[0,\alpha]\cup [1-\alpha,1]$,
\[
    \mc E_\pm(t,\sigma) \geq \frac{1}{2}\mc H^2(\Sigma_0).
\]
Then by the proof of Lemma \ref{lem:choose v} (taking $\epsilon_0<\mc H^2(\Sigma_0)/4$ therein), there exists $\beta\in(0,\alpha)$ such that for all $0\leq s\leq t\leq \beta$ or $1-\beta\leq s\leq t\leq 1$,
\begin{equation}\label{eq:st away from endpoints}
    \mc C_\pm(s,t,\sigma) < \frac{1}{4}\mc H^2(\Sigma_0)\leq \frac{1}{2}\mc E_\pm(x,\sigma), \text{ for all $x\in[s, t]$}.
\end{equation}
By the co-area formula, we have that
\[
     \sup_{\frac{1}{2}\beta\leq s<t\leq 1-\frac{1}{2}\beta} \frac{\mc C_{\pm} (s,t,\sigma)}{t-s} + \sup_{0<s<t<1} \mc C_\pm(s,t,\sigma) \to 0, \text{ as $\sigma\to 0$.}
\]
Hence we can take $\sigma_0>0$, such that
\begin{equation}\label{eq:small u}
     \mc C_\pm(s,t,\sigma)\leq \eta \cdot (t-s), \text{ for all $\frac{1}{2}\beta\leq s\leq t\leq 1-\frac{1}{2}\beta$ and $\sigma\in ( 0,\sigma_0]$},
\end{equation}
with equality holding if and only if $t=s$, and
\begin{equation}\label{eq:t-s large}
     \mc C_\pm(s,t,\sigma) < \frac{1}{2}\eta \beta, \text{ for all $0<s\leq t<1$ and $\sigma\in ( 0,\sigma_0]$}. 
\end{equation}

Let $h_0$ be the constant given by Lemma \ref{lem:choose v} using $\sigma_0$ here. Thus for all $0\leq s\leq t\leq 1$ with $|t-s|\leq h_0$, $1\geq \sigma\geq \sigma_0$,
\begin{equation}\label{eq:from lemma 3.7}
    \mc C_\pm(s,t,\sigma) < \frac{1}{2}\mc E_\pm(x,\sigma), \text{ for all $x\in (0,1)$}.
\end{equation}
Let
\[
    \Lambda_0=-\frac{\ln \sigma_0}{h_0}.
\]
It follows that $\sigma_0=e^{-\Lambda_0h_0}$. We now divide the desired inequality into three cases.
\begin{itemize}
    \item Suppose that $0<s\leq t\leq 1$ and $\sigma\in[ \sigma_0, e^{-\Lambda_0|t-s|}]$. Note that in this case, $e^{-\Lambda_0h_0}=\sigma_0\leq e^{-\Lambda_0|t-s|}$, which implies that 
\[
    |t-s|\leq h_0.
\]
Then by \eqref{eq:from lemma 3.7}, 
\begin{equation*}
    \mc C_\pm(s,t,\sigma) < \frac{1}{2}\mc E_\pm(x,\sigma), \text{ for all $x\in (0,1)$}.
\end{equation*}
\item Suppose that $0<s<t<1$, $\sigma\in(0,\sigma_0]$ and $(t-s)\geq \frac{1}{2}\beta$. Then by \eqref{eq:t-s large},
\[
    \mc C_\pm(s,t,\sigma) <\frac{1}{2}\eta \beta\leq \eta(t-s).
\]
\item Suppose that $0<s\leq t<1$, $\sigma\in(0,\sigma_0]$ and $(t-s)<\frac{1}{2}\beta$. Then we have the following subcases:
\begin{enumerate}
    \item $0\leq s\leq t\leq \beta$ or $1-\beta\leq s\leq t\leq 1$. Then by \eqref{eq:st away from endpoints},
    \[
    \mc C_\pm(s,t,\sigma)< \frac{1}{2}\mc E_\pm(x,\sigma), \text{ for all $x\in[s,t]$}.
    \]
    \item $\frac{1}{2}\beta\leq s\leq t\leq 1-\frac{1}{2}\beta$. Then the desired inequality \eqref{eq:small necks} follows from \eqref{eq:small u}. Moreover, if equality holds in \eqref{eq:small necks}, then $t=s$ by \eqref{eq:small u}. This further implies $\mc C_\pm(s,t,\sigma)=0$, and hence $\mc E_{\pm}(t, \sigma) = 0$. Thus $\sigma=0$ and $t = \frac{1}{2}$.
\end{enumerate}
\end{itemize}
This completes the proof of Lemma \ref{lem:small necks}.
\end{proof}
\begin{rmk}
   In the argument above, we fix $\eta$ initially, then choose $\alpha, \beta$, and $\sigma_0$ in turn, and finally define $\Lambda_0$ explicitly in terms of $\sigma_0$ and $h_0$.
\end{rmk}

We end this section by establishing an area bound for an interior slice together with its adjacent necks.
\begin{proposition}\label{prop:Ai Ci Ci-1 sum}
Given $\eta>0$, let $\Lambda_0>0$ be the constant in Lemma \ref{lem:small necks}. Let $\mk A_i$ and $\mk C_i$ be defined in \eqref{eq:A_i} and \eqref{eq:C_i} with respect to $\Phi_\Lambda: \Delta_k\to \oli{\ms E}$ for $\Lambda\geq \Lambda_0$. Then for all $1\leq i\leq k$, we have
\[
  \mk C_{i-1}  + \mk C_i +\mk A_i \leq \mc H^2(\Sigma_0)+ \eta\cdot (x_{i+1}-x_{i-1}),
\]
where we define $\mk C_0=\mk C_{k+1}:=0$, $x_0=0$, $x_{k+1}=1$.
Moreover, we also have
\[
    \mk C_{i-1} + \mk A_i\leq \mc H^2(\Sigma_0) + \eta\cdot (x_i-x_{i-1}),
\]
\[ 
   \mk C_{i}+ \mk A_i\leq \mc H^2(\Sigma_0) + \eta\cdot (x_{i+1}-x_i).
\]
The above inequalities are strict if $x_{i+1}>x_{i-1}$, $x_i>x_{i-1}$, or $x_{i+1}>x_i$, respectively.
\end{proposition}
\begin{proof}
Note that $\Lambda \geq \Lambda_0$, 
and 
\[
    \xi_i\leq e^{-\Lambda|x_{i+1}-x_i|}\leq e^{-\Lambda_0 |x_{i+1}-x_i| }.
    \]
Suppose that $i\in[1,k]$ is even and $x_{i+1}>x_{i-1}$. 
Applying Lemma \ref{lem:small necks} for $\sigma=\xi_{i-1}$, $s=x_{i-1}$, $t=x_i$ and $\sigma=\xi_i$, $s=x_i$  $t=x_{i+1}$ therein, we get
\begin{align*}
   &\ \ \ \  \mc C_{+}(x_{i-1},x_{i},\xi_{i-1}) + \mc C_{-}(x_i,x_{i+1},\xi_{i})\\
    &< \frac{1}{2}\mc E_+ (x_{i},\xi_{i-1}) + \frac{1}{2}\mc E_- (x_{i},\xi_i)  +\eta\cdot (x_{i}-x_{i-1}) +\eta\cdot (x_{i+1}-x_i)\\
    &\leq \mc H^2(\Sigma_0) -\frac{1}{2} \left[\mc H^2(\Upsilon(x_i)\cap \mc D(\xi_{i-1})) +\mc H^2(\Upsilon(x_i)\setminus \mc D(1-\xi_{i})  ) \right]+\eta\cdot (x_{i+1}-x_{i-1})\\
    &\leq \mc H^2(\Sigma_0)- \mc H^2(\Upsilon(x_i) \cap \mathcal{D}(\xi_{i-1}) \setminus \mathcal{D}(1-\xi_i)) +\eta\cdot (x_{i+1}-x_{i-1}).
\end{align*}
Note that the first inequality is strict as either $x_{i+1}>x_i$ or $x_i>x_{i-1}$. 
Thus,
\begin{equation*}
    \mk C_{i-1}+\mk C_i + \mk A_i \leq \mc H^2(\Sigma_0) + \eta\cdot (x_{i+1}-x_{i-1}).
\end{equation*}
The other cases can be proved by similar arguments.
\end{proof}

\section{Properties of $\Phi_\Lambda$}\label{sec:area upper bound}

\subsection{Area upper bounds}
By taking $\sigma_0=\frac{1}{3}$ in \eqref{eq:E+ lower bd for all t} and \eqref{eq:E- lower bd for all t}, there exists $\eta>0$ such that 
\begin{equation}\label{eq:choice of eta0}
    \mc E_\pm(t, \sigma) > 2\eta, \text{ for all $0\leq t\leq 1$ and $1\geq\sigma\geq \frac{1}{3}$}.
\end{equation}
Fix $\eta$ small enough. Let $\Lambda_0$ be the constant given in Lemma \ref{lem:small necks} associated with $\eta$. Given $\Lambda\geq \Lambda_0$ and a positive integer $k$, let $\Phi_\Lambda=(\Phi_\Lambda^1,\Phi_\Lambda^2):\Delta_k\to \oli{\ms E}$ be the continuous map given in Definition \ref{def:Phi_Lambda}. 

In this section, we prove that the area of $\Phi_\Lambda^1$ is uniformly bounded (for all $\Lambda\geq \Lambda_0$) for any given $k$.
For any $1\leq \ell\leq m\leq k$, denote by $\mk F_{\ell,m}\subset \Delta_k$ the subcomplex 
\begin{equation}\label{eq:F_lm}
    \mk F_{\ell,m}:= \{(x_1,\cdots,x_k)\in \Delta_k: x_i=0 \text{ for all } i<\ell, \text{ and } x_j=1 \text{ for all } j>m\}.
\end{equation}
Clearly, $\mk F_{1,k}=\Delta_k$ and $\mk F_{\ell+1,m},\mk F_{\ell, m-1}\subset \mk F_{\ell,m}$.
\begin{lemma}\label{lem:bounded by kW}
Let $\eta$ and $\Lambda_0$ be the constants as above. Then for each positive integer $k$ and $\Lambda\geq \Lambda_0$, letting $\Phi_\Lambda=(\Phi_\Lambda^1,\Phi_\Lambda^2):\Delta_k\to \oli{\ms E}$ be the map constructed in Definition \ref{def:Phi_Lambda} associated with $\Lambda,k$, we have
\[
    \mc H^2( \Phi_\Lambda^1(x)) < k \mc H^2(\Sigma_0) +\eta , \quad  \text{ for all } x\in \Delta_k.
\]
Moreover, we have
\begin{enumerate}
    \item for any $x\in \mk F_{\ell,m}\subset \Delta_k$,
    \[
          \mc H^2(\Phi_\Lambda^1(x)) < (m-\ell+1)\mc H^2(\Sigma_0)+\eta;
    \]

    \item for any $x\in \partial \mk F_{\ell,m}\setminus (\mk F_{\ell+1,m}\cup \mk F_{\ell,m-1})$,
    \[
          \mc H^2( \Phi_\Lambda^1(x))< (m-\ell) \mc H^2(\Sigma_0) -\eta.
    \]
\end{enumerate}
\end{lemma}

\begin{proof}
Consider $x=(x_1, \cdots, x_k) \in \Delta_k$. Define $\xi_0=\xi_k=0$ as an amendment for \eqref{eq:xilambda}. 
Recall that we have
\begin{equation}\label{eq:area of Phi}
  \mc H^2(\Phi_{\Lambda}^1) = \sum_{i=1}^{k-1} \mk C_i+ \sum_{i=1}^k \mk A_i.
\end{equation}
By Proposition \ref{prop:Ai Ci Ci-1 sum}, we have
\begin{equation}\label{eq:Ci-1 Ci and Ai}
    \mk C_{i-1}+\mk C_i + \mk A_i \leq \mc H^2(\Sigma_0) + \eta\cdot (x_{i+1}-x_{i-1}).
\end{equation}
\begin{equation}\label{eq:Ci-1 and Ai}
   \mk C_{i-1}+\mk A_i \leq \mc H^2(\Sigma_0) + \eta\cdot (x_i-x_{i-1}),
\end{equation}
\begin{equation}\label{eq:Ci and Ai}
    \mk C_i+\mk A_i \leq \mc H^2(\Sigma_0) + \eta\cdot (x_{i+1}-x_i).
\end{equation}
Moreover, the inequalities are strict if $x_{i+1}>x_{i-1}$, $x_i>x_{i-1}$, or $x_{i+1}>x_i$, respectively.
Plugging \eqref{eq:Ci and Ai} into \eqref{eq:area of Phi}, we obtain
\begin{align*}
 \mc H^2(\Phi_\Lambda^1) = \sum_{i=1}^k (\mk C_i+\mk A_i)\leq \sum_{i=1}^{k} \big[\mc H^2(\Sigma_0) + \eta \cdot (x_i - x_{i-1} ) \big]< k\mc H^2(\Sigma_0) + \eta .
\end{align*}

Next we consider $x = (x_1,\cdots, x_k)\in \mk F_{\ell, m}$. Plugging \eqref{eq:Ci-1 Ci and Ai}, \eqref{eq:Ci-1 and Ai} and \eqref{eq:Ci and Ai} into \eqref{eq:area of Phi}, we deduce that
\begin{align*}
    \mc H^2(\Phi_\Lambda^1) &= \sum_{i=\ell-1}^m \mk C_i+ \sum_{i=\ell}^m\mk A_i\\
    &= \mk C_{\ell-1} +\mk C_\ell + \mk A_\ell + \sum_{i=\ell+1}^m (\mk C_i+\mk A_i)\\
    &< \mc H^2(\Sigma_0) + \eta\cdot (x_{\ell+1}-x_{\ell-1}) + \sum_{i=\ell+1}^m \big[\mc H^2(\Sigma_0)+\eta\cdot (x_{i+1}-x_i)\big]\\
    &=(m-\ell+1)\mc H^2(\Sigma_0)+\eta,
\end{align*}
where we used that $x_{\ell-1}=0$ and $x_{m+1}=1$.

Finally we consider $x = (x_1,\cdots, x_k) \in \partial \mk F_{\ell,m}\setminus (\mk F_{\ell+1,m}\cup \mk F_{\ell,m-1})$. Then there exists $\ell \leq j\leq m-1$ such that 
$x_j=x_{j+1}$. Hence 
\[
    \xi_{j}= e^{ -\Lambda |x_{j+1}-x_j|} (1-\xi_{j-1}) =1-\xi_{j-1}.
    \]
It follows that $\mk A_j=0$. Moreover, we have that either $\xi_{j-1}\geq \frac{1}{2}$ or $\xi_j\geq \frac{1}{2}$. Then by \eqref{eq:choice of eta0}, 
\[
    \mc E_{\pm} (x_{j-1},\xi_{j-1} ) > 2\eta, \text{ or } \mc E_\pm (x_{j}, \xi_j) > 2\eta.
\]
We prove the desired inequality for
\[
    \text{ $j$ is even and $\mc E_\pm (x_{j}, \xi_j) > 2\eta$},
\]
and the remaining cases can be proved by a similar argument.
Since $x_j=x_{j+1}$, we have
\begin{align*}
    \mk A_{j+1}&= \mc H^2(\Upsilon(x_{j+1}) \cap \mc D (\xi_{j+1}) \setminus \mc D(1-\xi_j) )\\
    &=\mc H^2(\Upsilon(x_{j}) \cap \mc D (\xi_{j+1}) \setminus \mc D(1-\xi_j) )\\
    &\leq \mc H^2(\Upsilon(x_{j}) \setminus \mc D(1-\xi_j) )\\
    &= \mc H^2(\Sigma_0) - \mc E_-(x_j,\xi_j)\\
    &< \mc H^2(\Sigma_0) -2\eta.
\end{align*}
Plugging \eqref{eq:Ci-1 Ci and Ai}, \eqref{eq:Ci-1 and Ai}, \eqref{eq:Ci and Ai} and the above inequality into \eqref{eq:area of Phi} again, we deduce that
\begin{align*}
    \mc H^2(\Phi_\Lambda^1) &= \sum_{i=\ell-1}^m \mk C_i + \sum_{i=\ell}^m\mk A_i\\
    &= \mk C_{\ell-1}+ \mk C_\ell + \mk A_\ell + \sum_{i=\ell+1}^{j-1} (\mk C_i +\mk A_i) + \mk C_j+ \mk A_j+ \mk A_{j+1}+\sum_{i=j+1}^{m-2}(\mk C_i+\mk A_{i+1})\\
    &+ \mk C_{m-1} + \mk C_m + \mk A_m\\
    &< \mc H^2(\Sigma_0) + \eta\cdot(x_{\ell+1}-x_{\ell-1}) + (j-\ell-1) \mc H^2(\Sigma_0) + \eta\cdot(x_j-x_{\ell+1}) + \mc H^2(\Sigma_0)-2\eta \\
    &+ (m-j-2) \mc H^2(\Sigma_0) + \eta\cdot (x_{m-1}- x_{j+1}) +  \mc H^2(\Sigma_0) + \eta\cdot (x_{m+1}- x_{m-1}) \\
    &\leq (m-\ell) \mc H^2(\Sigma_0)-\eta,
\end{align*}
where we used $x_j=x_{j+1}$, $x_{\ell-1}=0$, $x_{m+1}=1$.

This finishes the proof of Lemma \ref{lem:bounded by kW}.
\end{proof}

We remark that $\eta$ can be arbitrarily small and it follows that for any fixed $k>1$,
\begin{equation}\label{eq:limit bound}
    \lim_{\Lambda\to\infty} \sup_{x\in \partial\mk F_{\ell,m}} \mc H^2(\Phi_\Lambda^1(x)) \leq (m-\ell)\mc H^2(\Sigma_0).
\end{equation}

Note that when $k$ is sufficiently large, there always exist two coordinates close to each other for any point in $\Delta_k$. Hence by our construction, a substantial amount of area overlaps and is omitted between the two corresponding slices, ensuring the area upper bound remains strictly less than $k\mc H^2(\Sigma_0)$.
\begin{lemma} \label{lem:choose N}
Let $\eta$ and $\Lambda_0$ be the constants as above. Then for every $\Lambda\geq \Lambda_0$, there exists $N$ large enough such that for all $k>N$,
\[
    \mc H^2(\Phi^1_\Lambda(x)) < (k-1) \mc H^2(\Sigma_0) + 2\eta, \quad  \text{ for all } x\in \Delta_k. 
\]
\end{lemma}
\begin{proof}
Note that 
\[
   \sup\{ \mc H^2\big(\Upsilon(t)\cap \mc D(\sigma )\setminus \mc D(1-\epsilon-\sigma)\big): t\in(0,1) , \sigma \in(0,1-\epsilon)\} \to 0, \text{ as $\epsilon\to 0^+$}.
\]
Then there exists $\delta>0$ such that for all $t\in (0,1)$ and $\sigma,\sigma'\in (0,1)$ with $1\geq \sigma+\sigma'\geq 1-\delta$, 
\begin{equation}\label{eq:choice of delta}
    \mc H^2( \Upsilon(t)\cap \mc D(\sigma) \setminus \mc D(1-\sigma') ) < \eta.
\end{equation}
Now we take $N$ sufficiently large such that 
\[
    N> -\frac{\Lambda}{\ln (1-\delta)}+1,
\]
which implies that
\[
    e^{-\Lambda/(N-1)}>1-\delta.
\]
Then for any $k>N$ and $x = (x_1,\cdots, x_k)\in \Delta_k$, there exists $1\leq j\leq  N-1$ such that 
\[
    d_j:=x_{j+1}-x_j \leq  \frac{1}{k-1} <  \frac{1}{N-1},
\]
which implies that 
\[   e^{-\Lambda d_j} > e^{-\Lambda/(N-1)} \geq 1-\delta.\]
By the definition of $\xi_i$, we then obtain that
\[
    \xi_{j-1}+\xi_{j} =\xi_{j-1}+ e^{-\Lambda d_j}(1-\xi_{j-1})> 1-\delta.
\]
By the choice of $\delta$ and \eqref{eq:choice of delta}, we have that 
\begin{equation}\label{eq:Aj bounded by eta}
     \mk A_j< \eta,
\end{equation}
where $\mk A_j$ is defined in \eqref{eq:A_i}. Plugging \eqref{eq:Aj bounded by eta}, \eqref{eq:Ci-1 Ci and Ai}, \eqref{eq:Ci-1 and Ai} and \eqref{eq:Ci and Ai} into \eqref{eq:area of Phi}, we deduce that
\begin{align*}
    \mc H^2(\Phi_\Lambda^1) &= \sum_{i=1}^{k-1} \mk C_i+ \sum_{i=1}^k \mk A_i\\
    &=\sum_{i=1}^{j-1}(\mk C_i+\mk A_i) + \mk A_j + \sum_{i=j}^{k-1} (\mk C_i+\mk A_{i+1})\\
    &< (j-1)\mc H^2(\Sigma_0) +\eta\cdot x_j + \eta +(k-j) \mc H^2(\Sigma_0)+ \eta\cdot (x_k-x_{j})\\
    &\leq (k-1) \mc H^2(\Sigma_0) +2\eta.
\end{align*}
This finishes the proof of Lemma \ref{lem:choose N}.
\end{proof}

\subsection{A deformation lemma}
Given $\Lambda_1>0$, if the maximal area of $\Phi_{\Lambda_1}:\Delta_k\to \oli{\ms E}$ is strictly less than $k\mc H^2(\Sigma_0)$, then for any $\Lambda\geq \Lambda_1$, we can always homotopically deform $\Phi_\Lambda$ (fixing the value on $\partial \Delta_k$) to a map with maximal area strictly less than $k\mc H^2(\Sigma_0)$.  
\begin{lemma}\label{lem:bounded by kW for all small r}
Let $\eta$ and $\Lambda_0$ be the constants as above. Suppose that $\Lambda_1\geq \Lambda_0$ and $k$ are such that
\[
   \mc H^2(\Phi^1_{\Lambda_1}(x))<(k-1)\mc H^2(\Sigma_0)+2\eta, \quad \forall x\in \Delta_k.
\]
Then for every $\Lambda\geq \Lambda_1$, there exists a continuous map $H:\Delta_k\times [0,1] \to \oli{\ms E}$ in the sense of $\ms F$-topology such that 
\begin{enumerate}
    \item $H(\cdot, 0)=\Phi_{\Lambda}(\cdot)$;
    \item $H(\cdot,t)=\Phi_\Lambda(\cdot)$ on $\partial \Delta_k$;
    \item $H(\cdot,1)$ satisfies all properties in Proposition \ref{lem:continuity of Phi_Lambda};
    \item $\mc H^2(H(\cdot, 1))<(k-1)\mc H^2(\Sigma_0)+2\eta$ in $\Delta_k$.
\end{enumerate}
\end{lemma}
\begin{proof}
Fix a $\Lambda\geq \Lambda_1$ and $k$.
Recall that by Lemma \ref{lem:bounded by kW}, for all $\lambda\geq \Lambda_0$, 
\[
     \mc H^2(\Phi^1_{\lambda}(x)) < (k-1) \mc H^2(\Sigma_0) +\eta, \quad \forall\, x\in \partial \Delta_k.  
\]
Then by Proposition \ref{lem:continuity of Phi_Lambda} (3), there exists a neighborhood $U$ of $\partial \Delta_k$, such that 
\[
    \mc H^2(\Phi^1_\lambda(x)) < (k-1)\mc H^2(\Sigma_0)+\eta, \quad \text{ for all $\lambda \in [\Lambda_1, \Lambda]$ and $x\in U$}.
\]
Now we take a smaller neighborhood $V\subset U$ of $\partial \Delta_k$ with $\oli{V}\subset\subset U$. Let $\zeta:\Delta_k\to [0,1]$ be a continuous cut-off function such that 
\[
     \zeta  = 0 \text{ in $V$,} \quad \zeta =1 \text{ in $\Delta_k\setminus U$}.
\]
Now we define 
\[
     H(x,t)= \Phi_{t\zeta \Lambda_1+(1-t\zeta)\Lambda}(x).
\]
Note that $\Phi_\Lambda(x)$ depends on $\Lambda, x$ continuously in the $\ms F$-topology by Proposition \ref{lem:continuity of Phi_Lambda} (3). Hence $H$ is continuous in the $\ms F$-topology. It remains to verify that $H$ satisfies the desired properties. 
Clearly, $H(\cdot, 0)= \Phi_\Lambda(\cdot)$. Since $\zeta=0$ on $\partial \Delta_k$, we know 
\[
 H(x,t) =\Phi_\Lambda(x), \text{ for all $x\in \partial \Delta_k$}.
\]
By Proposition \ref{lem:continuity of Phi_Lambda} (1) and the definition of $H$ as above, for $x\in \Delta_k\setminus \partial \Delta_k$, $H(x,t)$ consists of $(2k-1)$ smooth embedded surfaces with boundaries and vary continuously with respect to $x$ and $t\in [0,1]$ in the $C^\infty$ topology. Moreover, $\{H(x,1)\}_{x\in \Delta_k}$ satisfies all properties in Proposition \ref{lem:continuity of Phi_Lambda}. This verifies the third item. 
Note that 
\[
    \Lambda\geq  t\zeta \Lambda_1+ (1-t\zeta) \Lambda \geq \Lambda_1.
\]
By the choice of $U$, we have that  
\begin{equation}\label{eq:area bound in U}
   \mc H^2(H(x, t)) < (k-1)\mc H^2(\Sigma_0) + \eta \quad \text{ for all } x\in U.
\end{equation}
Observe that $\zeta=1$ outside $U$. Hence for all $x\in \Delta_k\setminus U$,
\[
    \mc H^2(H(x, 1)) =\mc H^2(\Phi_{\Lambda_1}(x)) < (k-1)\mc H^2(\Sigma_0)+2\eta.
\]
This together with \eqref{eq:area bound in U} verifies the last item. 
Hence Lemma \ref{lem:bounded by kW for all small r} is proved.
\end{proof}

\subsection{Properties of boundary maps}

For $t\in[0,1]$ and an integer $m>0$, we use $t^m$ to denote $(t,\cdots,t)\in[0,1]^m$.

\begin{lemma}\label{lem:mf F to flat F}
Let $\eta$ and $\Lambda_0$ be the given constants as above. Given an integer $k>1$ and $\theta>0$, there exist $\tau_{k, \theta}>0$, $\Lambda_{k, \theta}>0$ such that for $\Lambda\geq \Lambda_{k, \theta}$ and $x\in \mk F_{\ell,m}$, 
\[
    \mf F\left(\Phi_\Lambda^1(x), (m-\ell+1)[\Sigma_0]\right)<\tau_{k, \theta} \Longrightarrow \mc F\left(\Phi_\Lambda^2(x),(m+1)\llbracket S^3\rrbracket+(m-\ell+1)\llbracket \Omega(\frac{1}{2})\rrbracket\right)<\theta.
\]
\end{lemma}
\begin{proof}
Suppose the conclusion is not true; then there exist $\theta>0$, $k>1$ and sequences of $\tau_j\to 0$, $\lambda_j\to \infty$, $1\leq \ell_j \leq m_j \leq k$, $q_j=(q_j^1,\cdots, q_j^k)\in \mk F_{\ell_j,m_j}\subset \Delta_k$ such that 
\begin{equation}\label{eq:FF assumption}
     \mf F(\Phi_{\lambda_j}^1(q_j),(m-\ell+1)[\Sigma_0])<\tau_j, \quad   \mc F\left(\Phi_{\lambda_j}^2(q_j),(m+1)\llbracket S^3\rrbracket+(m-\ell+1)\llbracket \Omega(\frac{1}{2})\rrbracket\right) \geq \theta.
\end{equation}
Since there are only finitely many choices for $(\ell_j, m_j)$, up to a subsequence, we may assume that $\ell_j = \ell$ and $m_j = m$ for some $1\leq \ell \leq m\leq k$.
\begin{claim}
    We have that $q_j\to (0^{\ell-1}, \frac{1}{2},\cdots, \frac{1}{2}, 1^{k-m})$ as $j\to \infty$.
\end{claim}
\begin{proof}[Proof of Claim]
Suppose not. Then there exists $\delta>0$ such that $q_j^m>\frac{1}{2}+ \delta$ or $q_j^\ell<\frac{1}{2}-\delta$ for all $j$. Without loss of generality, we assume that $q_j^\ell<\frac{1}{2}-\delta$ for all $j$.

By plugging \eqref{eq:Ci-1 Ci and Ai}, \eqref{eq:Ci-1 and Ai} and \eqref{eq:Ci and Ai} into \eqref{eq:area of Phi}, we obtain that
\begin{align*}
  &\ \ \ \  \mc H^2\Big( \Phi_{\lambda_j}^1(q_j) \cap \left[\Omega_{\frac{1}{2}+\delta}\setminus \oli{\Omega_{\frac{1}{2}-\delta}}\right]\Big) \leq \sum_{i=\ell}^m \mk C_i +\sum_{i=\ell+1}^m \mk A_i\\
  &\leq \mk C_\ell +\mk C_{\ell+1}+ \mk A_{\ell+1} + \sum_{i=\ell+2}^{m} (\mk C_i+\mk A_i)\\
  &\leq \mc H^2(\Sigma_0) + \eta \cdot (x_{\ell+2}-x_\ell)+ \sum_{i=\ell+2}^m\Big( \mc H^2(\Sigma_0)+ \eta\cdot( x_{m+1}-x_{\ell+2} )\Big)\\
  &\leq (m-\ell)\mc H^2(\Sigma_0) +\eta.
\end{align*}
This contradicts that $\Phi_{\lambda_j}^1(q_j)$ converges to $(m-\ell+1)[\Sigma_0]$ in the sense of varifolds as $j\to\infty$, and the claim is proved. 
\end{proof}
To continue the proof, applying Lemma \ref{lem:flat convergence of Phi}, and reducing the coefficients modulo 2, we have  
\[
    \mc F\left(\Phi_{\lambda_j}^2( q_j), (k-m+k-1)\llbracket S^3\rrbracket+(m-\ell+1)\llbracket \Omega(\frac{1}{2})\rrbracket\right) \to 0,
\] 
as $j\to \infty$, which contradicts \eqref{eq:FF assumption}. 
Therefore, Lemma \ref{lem:mf F to flat F} is proved.
\end{proof}

\section{Smooth families}\label{sec:smooth family}
Recall that we have constructed maps $\Phi_\Lambda:\Delta_k\to\oli{\ms E}$. In order to apply the Simon-Smith min-max theory, we will modify the maps (and also the domains) to be new maps whose images are contained in $\ms E$.

Recall that 
\[
    \Delta_k=\{ (x_1,\cdots, x_k)\in[0,1]^k;x_1\leq \cdots\leq x_k\}.
\]
The barycenter $\bm c$ is the average of the vertices:
\[
  \bm  c = \left( \frac{1}{k+1}, \frac{2}{k+1}, \dots, \frac{k}{k+1} \right) .
   \]
Given a small constant $\nu\in(0,1)$, define the retraction map $\rho_{\nu}:\Delta_k\to \Delta_k$ by
\[
    \rho_\nu (x)= (1-\nu) x + \nu \bm c.
\]
Then one can verify that $\Delta_k^\nu:=\rho_\nu(\Delta_k)\subset \Delta_k\setminus \partial \Delta_k$ and
\[
    \Delta_k^\nu =\left \{ (x_1,\cdots,x_k) \in \Delta_k; x_1 \ge \frac{\nu}{k+1},\; x_{i+1} - x_i \ge \frac{\nu}{k+1},\; 1 - x_k \ge \frac{\nu}{k+1} \right\}. 
\]

Let $\eta>0$ be a fixed small constant, and $\Lambda_0$ be the constant given in Lemma \ref{lem:small necks}. 
By Lemma \ref{lem:choose N}, we can take $k_0$ large enough, 
such that the map $\Phi_{\Lambda_0}:\Delta_{k_0}\to \oli{\ms E}$ given in Definition \ref{def:Phi_Lambda} satisfies
\[
    \sup_{x\in \Delta_{k_0}}\mc H^2(\Phi^1_{\Lambda_0}(x)) < (k_0-1)\mc H^2(\Sigma_0)+2\eta. 
\]    
Fix the $\Lambda_0$ and $k_0$. By Lemma \ref{lem:bounded by kW}, we have for all $\Lambda\geq \Lambda_0$, $x\in \mk F_{\ell,m}\subset \Delta_{k_0}$, 
\begin{equation}\label{eq:m-l+1W+eta}
    \mc H^2(\Phi^1_\Lambda(x)) < (m-\ell+1)\mc H^2(\Sigma_0)+\eta,
\end{equation}
and for all $x\in \partial \mk F_{\ell,m}\setminus (\mk F_{\ell+1,m}\cup \mk F_{\ell,m-1})$,  
\begin{equation}\label{eq:m-l-eta}
    \mc H^2( \Phi^1_\Lambda(x))< (m-\ell)\mc H^2(\Sigma_0) - \eta.
\end{equation}
Moreover, by Lemma \ref{lem:bounded by kW for all small r}, $\Phi_\Lambda: \Delta_{k_0}\to \oli{\ms E}$ can always be homotopically deformed to be another continuous map $ \Phi'_\Lambda=( \Phi'^1_\Lambda,\Phi'^2_\Lambda):\Delta_{k_0}\to\oli{\ms E}$ with $\Phi'_\Lambda = \Phi_\Lambda$ on $\partial \Delta_{k_0}$, and
\begin{equation}\label{eq:wti r and k0}
    \mc H^2(\Phi'^1_\Lambda) < (k_0-1)\mc H^2(\Sigma_0)+2\eta, \text{ for all $x\in\Delta_{k_0}$}.
\end{equation}
Since $\Phi_\Lambda'$ is continuous in the $\ms F$ topology on $\Delta_{k_0}$, for each $\Lambda>0$, we can take $\nu = \nu_\Lambda>0$ such that 
\begin{equation}\label{eq:small nu}
    \sup_{x\in \Delta_{k_0}}\ms F \big(\Phi'_\Lambda(\rho_\nu (x)), \Phi'_\Lambda (x)\big) < \frac{1}{2\Lambda}.
\end{equation}

Next, for the given $\Lambda>0$ and $\nu>0$, since $\Phi'_\Lambda(x)|_{\Delta_{k_0}^\nu}$ are all piecewise-smooth embedded spheres (Proposition \ref{prop:2k-1 pieces} (1) and Lemma \ref{lem:bounded by kW for all small r}), by Proposition \ref{prop:mollification} and Remark \ref{rmk:unifor perturb}, we can perturb $\Phi'_\Lambda:\Delta_{k_0}^\nu\to \oli{\ms E}$ to $\widehat \Phi_\Lambda=({\widehat\Phi}_\Lambda^1,\widehat \Phi_\Lambda^2): \Delta_{k_0}^\nu\to\ms E$ such that 
\begin{equation}\label{eq:small peturb}
    \sup_{x\in \Delta_{k_0}}\ms F\big( \widehat \Phi_\Lambda(x), \Phi_\Lambda'(x) \big) <\frac{1}{2\Lambda}.
\end{equation}
Define $\wti \Phi_\Lambda: \Delta_{k_0}\to \ms E$ by
\begin{equation}\label{eq:widetildePhi}
    \wti \Phi_\Lambda(x) = \widehat\Phi_\Lambda(\rho_\nu(x)).
\end{equation}
Then by \eqref{eq:small nu} and \eqref{eq:small peturb}, we have 
\begin{equation}\label{eq:ms F small perturb}
\begin{aligned}
    &\sup_{x\in \Delta_{k_0}} \ms F\big(\wti \Phi_\Lambda(x) , \Phi_\Lambda' (x)\big) \\
    & \leq \sup_{x\in \Delta_{k_0}} \ms F\big(\widehat \Phi_\Lambda(\rho_\nu(x)) , \Phi'_\Lambda (\rho_\nu(x))\big) +  \sup_{x\in \Delta_{k_0}} \ms F\big(\Phi'_\Lambda(\rho_\nu (x)), \Phi_\Lambda'  (x)\big) < \frac{1}{\Lambda}.
\end{aligned}
\end{equation}
Together with  \eqref{eq:m-l+1W+eta}, \eqref{eq:m-l-eta} and \eqref{eq:wti r and k0}, we conclude that there exists $\Lambda_2\geq \Lambda_0$ (depending on $\eta$) such that for all $\Lambda\geq \Lambda_2$,
\begin{equation}\label{eq:less than k0:smooth}
    \mc H^2(\wti\Phi^1_\Lambda(x)) < (k_0-1)\mc H^2(\Sigma_0) +3\eta, \text{ for all $x\in \Delta_{k_0}$} ,
\end{equation}
\begin{equation}\label{eq:m-l+1W+eta:smooth}
    \mc H^2(\wti\Phi^1_\Lambda(x)) < (m-\ell+1)\mc H^2(\Sigma_0)+ 2\eta, \text{ for all $x\in \mk F_{\ell,m}\subset \Delta_{k_0}$},
\end{equation}
\begin{equation}\label{eq:m-l-1W+eta:smooth}
    \mc H^2( \wti \Phi^1_\Lambda(x))< (m-\ell)\mc H^2(\Sigma_0) - \frac{1}{2}\eta, \text{ for all $x\in \partial \mk F_{\ell,m}\setminus (\mk F_{\ell+1,m}\cup \mk F_{\ell,m-1})$}.
\end{equation}

We now prove that $\wti\Phi_\Lambda$ also satisfies a similar property as in Lemma \ref{lem:mf F to flat F}. Note that $\Phi_\Lambda'$ may not be close to $\Phi_\Lambda$ away from $\partial \Delta_{k_0}$. Hence we can only obtain the following property on $\partial \Delta_{k_0}$. 

\begin{lemma}\label{lem:mf F small to mc F small:smooth}
Given $\theta>0$, there exist $\hat\tau_\theta > 0$, $\hat\Lambda_\theta > 0$, such that for all $\Lambda \geq \hat\Lambda_\theta$ and $x\in \mk F_{\ell,m}\subset \partial \Delta_{k_0}$,
\begin{equation}\label{eq:ms F to mc F:smooth}
\begin{aligned}
    \mf F\left(\wti\Phi_\Lambda^1(x), (m-\ell+1)[\Sigma_0]\right) & < \hat\tau_\theta \\
    & \Longrightarrow \mc F\left(\wti\Phi_\Lambda^2(x),(m+1)\llbracket S^3\rrbracket+(m-\ell+1)\llbracket \Omega(\frac{1}{2})\rrbracket\right)<\theta.
\end{aligned}  
\end{equation}
\end{lemma}
\begin{proof}
Recall that by Lemma \ref{lem:mf F to flat F}, given $\theta>0$ (with $k=k_0$ fixed), there exist $\Lambda_{k_0, \theta} > 0$ and $\tau_{k_0, \theta} > 0$ such that for all $\Lambda > \Lambda_{k_0, \theta}$ and $x\in \mk F_{\ell,m}\subset \Delta_{k_0}$,
\begin{equation}\label{eq:ms F to mc F}
\begin{aligned}
    \mf F\left(\Phi_\Lambda^1(x), (m-\ell+1)[\Sigma_0]\right) & < \tau_{k_0, \theta}\\
    &\Longrightarrow \mc F\left(\Phi_\Lambda^2(x),(m+1)\llbracket S^3\rrbracket+(m-\ell+1)\llbracket \Omega(\frac{1}{2})\rrbracket\right)<\theta/2.
    \end{aligned}
\end{equation}
Recall that $\Phi_\Lambda'=\Phi_\Lambda$ on $\partial \Delta_{k_0}$. This together with \eqref{eq:ms F small perturb} implies that for $x\in \mk F_{\ell,m}\subset \partial\Delta_{k_0}$ and $\Lambda >\max\{\frac{2}{\tau_{k_0, \theta}}, \frac{2}{\theta}\}$,
\begin{align*}
    \mf F\left(\wti \Phi_\Lambda^1(x), (m-\ell+1)[\Sigma_0]\right) & < \frac{1}{2}\tau_{k_0, \theta} \\ & \Longrightarrow  \mf F\left( {\Phi'^1_\Lambda}(x), (m-\ell+1)[\Sigma_0]\right) \leq \frac{1}{2}\tau_{k_0, \theta} + \frac{1}{\Lambda} < \tau_{k_0, \theta} \\
    &\Longrightarrow \mf F\left( {\Phi^1_\Lambda}(x), (m-\ell+1)[\Sigma_0]\right) < \tau_{k_0, \theta}\\
\text{ by \eqref{eq:ms F to mc F}}   
    &\Longrightarrow \mc F\left(\Phi_\Lambda^2(x),(m+1)\llbracket S^3\rrbracket + (m-\ell+1)\llbracket \Omega(\frac{1}{2})\rrbracket\right)<\theta/2\\
\text{ by \eqref{eq:ms F small perturb}}     
    &\Longrightarrow \mc F\left( \wti\Phi^2_\Lambda(x),(m+1)\llbracket S^3\rrbracket + (m-\ell+1)\llbracket \Omega(\frac{1}{2})\rrbracket\right)<\theta/2 + \frac{1}{\Lambda} < \theta.
\end{align*}
Letting $\hat\Lambda_\theta := \max\{\Lambda_{k_0, \theta},\frac{2}{\tau_{k_0, \theta}}, \frac{2}{\theta}\}$ and $\hat\tau_\theta = \frac{1}{2}\tau_{k_0, \theta}$, we can deduce \eqref{eq:ms F to mc F:smooth}. 
\end{proof}

\section{Relative sweepouts on simplexes}\label{sec:apply min-max}

Under the assumption for proof by contradiction that $(S^3, g)$ admits only one embedded minimal 2-sphere $\Sigma_0$, we introduce a procedure for iterative relative min-max constructions in $\ms E$ so as to derive a contradiction at the end.

\subsection{A topological lemma}
The following topological lemma will be used in the Lusternik-Schnirelmann type arguments in Section \ref{SS:iterative min-max}.

\begin{lemma}\label{lem:top}
Let $X$ be a cubical subcomplex such that 
\[
    \partial X=\partial \Delta_k.
\]
Denote by $Y\subset X$ the subcomplex such that 
\begin{enumerate}
    \item\label{item:only intersects top and side} $Y$ does not intersect the closure of $\partial X\setminus (\mk F_{1,k-1}\cup \mk F_{2,k})$;
    \item\label{item:separation} $Y$ can be decomposed into two disjoint subsets $Y_1$ and $Y_2$ such that $Y_1\cap \mk F_{1,k-1}=\emptyset$ and $Y_2\cap \mk F_{2,k}=\emptyset$.
\end{enumerate}
Then by taking a subdivision of $X$, there exists a subcomplex $Z\subset X\setminus Y$ with $\partial Z=\partial \mk F_{1,k-1}$.
\end{lemma}

\begin{proof}
Denote by $A:= X\setminus Y_1$ and $B:=X\setminus Y_2$. Since $Y_1\cap Y_2=\emptyset$, we have that $A\cup B = X$ and the relative interiors of $A$ and $B$ cover $X$. Observe that there is a homomorphism \cite{Hat-AT-book}*{Section 2.2, Page 150}
\[
   \partial_*: H_{k-1}(X;\mb Z) \longrightarrow H_{k-2}(A \cap B;\mb Z)
\]
defined as follows: a class $\alpha \in H_{k-1}(X;\mb Z)$ is represented by a cycle $z$, and we can choose $z$ to be a sum $x+y$ of chains in $A$ and $B$, respectively. Then the element $\partial_*\alpha \in H_{k-2}(A \cap B;\mb Z)$ is represented by the cycle $\partial x = -\partial y$.

Note that $\partial X$ can be decomposed into two pieces: $\mk F_{1,k-1} \subset A$ (since $Y_1 \cap \mk F_{1,k-1} = \emptyset$ by assumption \eqref{item:separation}) and the closure $\overline{\partial X\setminus \mk F_{1,k-1}} \subset B$. The latter inclusion holds because $Y_2$ avoids $\mk F_{2,k}$ by assumption \eqref{item:separation} and avoids the remainder of the boundary by assumption \eqref{item:only intersects top and side}. 

These two chains share the same boundary, $\partial \mk F_{1,k-1}$. Hence, by definition, we have
\[
     \partial_*([\partial X]) = [\partial \mk F_{1,k-1}] \in H_{k-2}(A\cap B;\mb Z).
\]
Clearly, $[\partial X]=0\in H_{k-1}(X;\mb Z)$ since it bounds $X$. It follows that $[\partial \mk F_{1,k-1}]=0\in H_{k-2}(A\cap B;\mb Z)$. This implies that after a subdivision of $X$(c.f. \cite{Hat-AT-book}*{Theorem 2.27}), $\partial \mk F_{1,k-1}$ bounds a subcomplex in $A\cap B = X\setminus Y$. 
\end{proof}

\subsection{Iterative relative min-max constructions}\label{SS:iterative min-max}

We start by adapting the setups for relative min-max constructions in Section \ref{SS:SS min-max} to the current setting. We can always embed $\Delta_{k_0}$ as a cubical subcomplex of some $I(m_0, n)$. Let $Z$ be a cubical subcomplex of $\partial \Delta_{k_0}$, for example, $Z = \partial \Delta_{k_0}$ or $\partial \mk F_{1, m}$. Given a continuous map $\Phi: Z \to \ms E$, we denote by $[\Phi; \rel Z]$ the collection of continuous maps $\Psi: X \to \ms E$ such that $\partial X = Z$ (where $X\subset I(m_0,n')$ is a cubical subcomplex for some $n'\geq n$) and $\Psi|_{\partial X = Z}=\Phi$. Denote by 
\[
    \mf L([\Phi; \rel Z]):= \inf_{\Psi\in [\Phi;\rel Z]} \sup_{x\in \mr{dom} (\Psi)} \mc H^2(\Psi(x)).
\]

Recall that $\Lambda_0$, $k_0$, $\Lambda_2$ are fixed in Section \ref{sec:smooth family}. Let $\theta_0$, $\tau_0$ be given in Proposition \ref{prop:no flipping near a varifold},  and $\hat\tau_{\theta_0}$, $\hat\Lambda_{\theta_0}$ be given in Lemma \ref{lem:mf F small to mc F small:smooth} associated with $\theta = \theta_0$.

For any $\Lambda\geq \Lambda_0$, recall that $\wti \Phi_\Lambda$ is given in \eqref{eq:widetildePhi}. For $\Lambda\geq \Lambda_2$, by \eqref{eq:less than k0:smooth}, we have
\begin{equation}\label{eq:less than k0}
    \mf L([\wti \Phi_\Lambda|_{\partial \Delta_{k_0}};\rel \partial \Delta_{k_0}])<(k_0-1)\mc H^2(\Sigma_0)+3\eta.
\end{equation}

\begin{proposition}\label{prop:Lm=mW}
Suppose that $(S^3,g)$ admits only one embedded minimal sphere $\Sigma_0$. Then 
\begin{itemize}
    \item[($*$)] there exists $\Lambda_3 \geq \max\{\Lambda_2, \hat\Lambda_{\theta_0}\}$ such that for all $1\leq m\leq k_0$ and $\Lambda\geq \Lambda_3$, we have
    \[ \mf L([\wti \Phi_{\Lambda}|_{\partial \mk F_{1,m}}; \rel \partial \mk F_{1,m}])= m\mc H^2(\Sigma_0).\]
\end{itemize}
\end{proposition}
\begin{proof}
We prove this inductively. 

\medskip
\noindent{\bf Step I: $m=1$}.
Note that $\mk F_{1,1}= [0,1]\times \{1^{k_0-1}\}$. By Lemma \ref{lem:flat convergence of Phi}, \eqref{eq:ms F small perturb} and the fact that $\Phi'_\Lambda = \Phi_\Lambda$ along $\partial \Delta_{k_0}$, we have
\[
   \lim_{\Lambda\to +\infty} \wti \Phi_\Lambda (0,1^{k_0-1})= \lim_{\Lambda\to \infty} \Phi_{\Lambda} (0,1^{k_0-1}) = (0, \emptyset),
   \]
\[
      \lim_{\Lambda\to +\infty} \wti \Phi_\Lambda (1,\cdots, 1)= \lim_{\Lambda\to \infty} \Phi_{\Lambda} (1,\cdots,1) = (0, \llbracket S^3\rrbracket), 
\]
in the sense of $\ms F$-topology. In particular,
\[
    \lim_{\Lambda\to+\infty } \mc F\big(\wti \Phi^2_\Lambda (0,1^{k_0-1})+ \wti \Phi^2_\Lambda (1^{k_0}),\llbracket S^3\rrbracket \big) =0.
\]    
Application of Proposition \ref{prop:no flipping near a varifold} gives $\tau_0>0$, such that 
for all sufficiently large $\Lambda$ and for all $\Psi=(\Psi^1,\Psi^2)\in [\wti \Phi_\Lambda|_{\partial \mk F_{1,1}};\rel \partial \mk F_{1,1}]$, (with $k=0$ therein,)
\[
   \sup_{x\in \mr{dom} (\Psi)} \mf F\big(\Psi^1(x),0\big) \geq \tau_0.
\]
It follows that $\mf L( [\wti \Phi_\Lambda|_{\partial \mk F_{1,1}};\rel \partial \mk F_{1,1} ]) >b_0$ for some constant $b_0>0$. Since $\wti \Phi^1_\Lambda|_{\partial \mk F_{1,1}} \to 0$ in the sense of varifold as $\Lambda\to+\infty$, by Theorem \ref{thm:relative min-max theorem},  
$\mf L( [\wti \Phi_\Lambda|_{\partial \mk F_{1,1}} ;\rel \partial \mk F_{1,1}])$ is realized by a stationary integral varifold supported on embedded minimal spheres for all sufficiently large $\Lambda$. For such $\Lambda$, it follows that $\mf L( [\wti \Phi_\Lambda|_{\partial \mk F_{1,1}} ;\rel \partial \mk F_{1,1}])\geq \mc H^2(\Sigma_0)$. This together with \eqref{eq:m-l+1W+eta:smooth} implies that
\[
    \mf L\left([\wti \Phi_{ \Lambda}|_{\partial \mk F_{1,1}};\rel \partial \mk F_{1,1}]\right) = \mc H^2(\Sigma_0).
\]
Hence ($*$) holds true for $m=1$.

\medskip
\noindent{\bf Step II: Suppose that ($*$) holds true for $m=1,\cdots, j<k_0$.} By \eqref{eq:m-l+1W+eta:smooth} \eqref{eq:m-l-1W+eta:smooth} and Theorem \ref{thm:relative min-max theorem}, we have for $\Lambda \geq \Lambda_2$
\[ 
   \mf L\left([\wti \Phi_{ \Lambda}|_{\partial \mk F_{1,j+1}};\rel \partial \mk F_{1,j+1}]\right)  \leq (j+1)\mc H^2(\Sigma_0).
\]

Suppose on the contrary that ($*$) does not hold for $m=j+1$, that is, there exists a sequence $\lambda_i\to \infty$, such that
\[
 \mf L\left([\wti \Phi_{\lambda_i}|_{\partial \mk F_{1,j+1}};\rel \partial \mk F_{1,j+1}]\right)  < (j+1)\mc H^2(\Sigma_0).
\]
Observe that 
\[
  \mf L\left([\wti \Phi_{\lambda_i}|_{\partial \mk F_{1,j+1}}; \rel \partial \mk F_{1,j+1}]\right) \geq \sup \{\mc H^2(\wti \Phi^1_{\lambda_i}(x));x\in \partial \mk F_{1,j+1}\} \geq 
  j\mc H^2(\Sigma_0).
\]
The last $\geq$ follows from the induction hypothesis since $\partial \mk F_{1,j+1}\supset \partial \mk F_{1,j} $ and hence $\wti\Phi_{\lambda_i}|_{\partial \mk F_{1,j+1}}$ contains an element in $[\wti \Phi_{\lambda_i}|_{\partial \mk F_{1,j}};\rel \partial \mk F_{1,j}]$.  
If the following holds 
\[
\mf L\left([\wti \Phi_{\lambda_i}|_{\partial \mk F_{1,j+1}};\rel \partial \mk F_{1,j+1}]\right) > \sup \{\mc H^2(\wti \Phi^1_{\lambda_i}(x));x\in \partial \mk F_{1,j+1}\},
\]
then by Theorem \ref{thm:relative min-max theorem}, $\mf L\left([\wti\Phi_{\lambda_i}|_{\partial \mk F_{1,j+1}};\rel \partial \mk F_{1,j+1}]\right)$ is realized by a stationary integral varifold supported on embedded minimal spheres, with total measure lying in $\big(j\mc H^2(\Sigma_0), (j+1)\mc H^2(\Sigma_0)\big)$. This contradicts with the assumption that $\Sigma_0$ is the only embedded minimal sphere. Thus we conclude that 
\[
    \mf L\left([\wti\Phi_{\lambda_i}|_{\partial \mk F_{1,j+1}};\rel \partial \mk F_{1,j+1}]\right) = \sup \{\mc H^2(\wti \Phi_{\lambda_i}^1(x));x\in \partial \mk F_{1,j+1}\} \to j\mc H^2(\Sigma_0), \text{ as $i\to\infty$}.
\]
The last $\to$ follows from \eqref{eq:limit bound} and \eqref{eq:ms F small perturb}.
Therefore there exists a sequence of complexes $X_i\subset I(m_0,n'_i)$ 
and continuous maps $\Psi_i: X_i\to  \ms E$ with $\partial X_i=\partial \mk F_{1,j+1}$, $\Psi_i|_{\partial X_i}= \wti \Phi_{\lambda_i}|_{\partial \mk F_{1,j+1}}$, such that
\begin{equation}\label{eq:bounded by jSigma0}
   \sup_{x\in X_i} \mc H^2( \Psi^1_i(x) ) \to j \mc H^2(\Sigma_0), \quad \text{ as } i\to\infty.
\end{equation}

Recall that $\tau_0$ and $\theta_0$ are the constants in Proposition \ref{prop:no flipping near a varifold}. Let $\tau'<\tau_0$ be small enough such that 
\begin{equation}\label{eq:mf F small to mf F bound}
    \|V\|(M) > j\mc H^2(\Sigma_0) -\eta/2, \text{ for all $V\in \mc V(S^3)$ with $\mf F(V,j[\Sigma_0])\leq \tau'$}.
\end{equation}
Let $\tau = \hat\tau_{\theta_0/2} <\tau'$ be the constant given in Lemma \ref{lem:mf F small to mc F small:smooth} by taking $\theta= \frac{1}{2}\theta_0$.
Let 
\[
    Y_i:=\{ x\in X_i: \, \mf F(\Psi_i^1(x), j[\Sigma_0]) <\tau \}.
\]
Note that $\mk F_{1,j}\subset \partial \mk F_{1,j+1}=\partial X_i$. We now claim that $Y_i\cap \mk F_{1,j}\neq \emptyset$ for all large $i$. Indeed, one can take $x_1<x_2<\cdots<x_{j}$ such that 
\[
   \mf F \left( \sum_{i=1}^{j} [\Upsilon(x_i)], j[\Sigma_0]\right)<\tau/2. 
\]
(Recall that $\Upsilon$ is the foliation given in the beginning of Section \ref{sec: construct family of spheres}.) 
Note that by Proposition \ref{prop:2k-1 pieces} (2) and \eqref{eq:ms F small perturb}, for $x_1<x_2<\cdots<x_{j}$, 
\[
    \lim_{i\to\infty}\wti \Phi_{\lambda_i}^1(x_1,\cdots, x_{j},1^{k_0-j}) =\lim_{i\to\infty} \Phi_{\lambda_i}^1( x_1,\cdots,x_j,1^{k_0-j})= \sum_{i=1}^{j} [\Upsilon(x_i)] 
\]
in the sense of varifolds. Hence for all large $i$,
\[
 \mf F\big( \Psi_i^1(x_1,\cdots, x_{j},1^{k_0-j}) , j[\Sigma_0]\big) =\mf F\big( \wti \Phi_{\lambda_i}^1(x_1,\cdots, x_{j},1^{k_0-j}) , j[\Sigma_0]\big) < \tau,
\]   
which implies that $(x_1,\cdots, x_{j},1^{k_0-j}) \in Y_i$ and hence
\[
    Y_i\cap \mk F_{1,j} \neq \emptyset. 
\]
Similarly, we also have that for all large $i$,
\[
     Y_i\cap \mk F_{2,j+1}\neq \emptyset. 
\]

By Lemma \ref{lem:mf F small to mc F small:smooth} and the fact that $\mk F_{1,j}\cup \mk F_{2,j+1}\subset\partial \Delta_{k_0}$ (because of $j<k_0$), we have
\[
x\in Y_i\cap \mk F_{1,j} \Longrightarrow \mc F\left(\wti \Phi_{\lambda_i}^2(x),(j+1)\llbracket S^3\rrbracket + j\llbracket \Omega(\frac{1}{2})\rrbracket\right) < \frac{1}{2}\theta_0,
\]
\[
y\in Y_i\cap \mk F_{2,j+1} \Longrightarrow \mc F\left( \wti\Phi_{\lambda_i}^2(y)),(j+2)\llbracket S^3\rrbracket + j\llbracket \Omega(\frac{1}{2})\rrbracket\right) < \frac{1}{2}\theta_0.
\]
Therefore, we obtain that for any $x\in Y_i\cap \mk F_{1,j}$ and $y\in Y_i\cap \mk F_{2,j+1}$, 
\begin{align*}
       &\ \ \ \ \mc F\big(\Psi_{i}^2(x)+\Psi_{i}^2(y), \llbracket S^3\rrbracket \big) \\
       &\leq \mc F\left(\wti \Phi_{\lambda_i}^2(x),(j+1)\llbracket S^3\rrbracket + j\llbracket \Omega(\frac{1}{2})\rrbracket\right) +  \mc F\left(\wti \Phi_{\lambda_i}^2(y),(j+2)\llbracket S^3\rrbracket + j\llbracket \Omega(\frac{1}{2})\rrbracket\right)\\
       &<\frac{1}{2}\theta_0 + \frac{1}{2}\theta_0 =\theta_0,
\end{align*}
where we used the fact that $\partial X_i=\partial \mk F_{1,j+1}$ and $\Psi_i|_{\partial X_i} = \wti \Phi_{\lambda_i}|_{\partial \mk F_{1,j+1}}$.

Hence by Proposition \ref{prop:no flipping near a varifold}, there does not exist any curve joining $Y_i\cap \mk F_{1,j}$ and $Y_i\cap \mk F_{2,j+1}$ in $Y_i$, that is, $Y_i$ can be decomposed into two disjoint subset $Y_i^1$ and $Y_i^2$ such that 
\[
    Y_i^1\cap \mk F_{1,j}=\emptyset \text{ and } Y_i^2\cap \mk F_{2,j+1}=\emptyset.
\]
Recall that by \eqref{eq:m-l-1W+eta:smooth}
\[
   \mc H^2( \Psi^1_i(x)) = \mc H^2( \wti \Phi_{\lambda_i}^1 (x)) \leq j\mc H^2(\Sigma_0) -\eta/2, \text{ for all $x\in \partial \mk F_{1,j+1}\setminus (\mk F_{1,j}\cup \mk F_{2,j+1})$}.
\]
This together with \eqref{eq:mf F small to mf F bound} implies that $Y_i$ does not intersect $\partial \mk F_{1,j+1}\setminus (\mk F_{1,j}\cup \mk F_{2,j+1})$. Now we have verified the conditions \eqref{item:only intersects top and side} and \eqref{item:separation} in Lemma \ref{lem:top} for $Y=Y_i$, $X=X_i$, $k=j+1$ therein. Hence by Lemma \ref{lem:top}, there exists a subcomplex $Z_i\subset X_i\setminus Y_i$ with $\partial Z_i=\partial \mk F_{1,j}$. Observe that $\Psi_i|_{Z_i}\in [\wti \Phi_{\lambda_i}|_{\partial \mk F_{1,j}}; \partial \mk F_{1,j}]$. 
Then by induction, we have that for $i$ large enough
\[
    \sup_{x\in {Z_i}} \mc H^2(\Psi_i^1(x)) \geq \mf L( [\wti \Phi_{\lambda_i}|_{\partial \mk F_{1,j}}; \partial \mk F_{1,j}]) = j\mc H^2(\Sigma_0).
\]
This together with \eqref{eq:bounded by jSigma0} gives that
\[
    \sup_{x\in {Z_i}} \mc H^2(\Psi_i^1(x)) \to j\mc H^2(\Sigma_0), \text{ as } i\to \infty.
\]
Recall that by \eqref{eq:m-l+1W+eta:smooth} and \eqref{eq:m-l-1W+eta:smooth}, for $x\in \partial Z_i= \partial \mk F_{1,j}$,
\[
    \mc H^2(\Psi_i^1(x))=\mc H^2(\wti \Phi_{\lambda_i}^1 (x) ) \leq (j-1)\mc H^2(\Sigma_0) + 2\eta.
\]
Therefore $\{\Psi_i: Z_i\to \ms E\}$ forms a minimizing sequence associated with boundary data $\{\wti \Phi_{\lambda_i}: \partial \mk F_{1,j} \to \ms E\}$ satisfying \eqref{eq:width nontrivial2}. 
Thus by Theorem \ref{thm:min-max thm for sequences}, we conclude that there exists a subsequence $\{y_{i'}\in Z_{i'}\}$ such that 
\[
    \mf F(\Psi_{i'}^1(y_{i'}), j[\Sigma_0]) \to 0, \text{ as $i'\to \infty$}.
\]
This contradicts the choice of $Y_i$ and $Z_i$. This finishes the induction and completes the proof of Proposition \ref{prop:Lm=mW}.
\end{proof}

\subsection{Main result}
We are now ready to prove the main result. 
\begin{theorem}\label{thm:two spheres}
Every Riemannian 3-sphere contains at least two distinct embedded minimal spheres.
\end{theorem}
\begin{proof}
Suppose not; then by Proposition \ref{prop:Lm=mW}, we have that for all $\Lambda>\Lambda_3$,
\[
    \mf L([\wti \Phi_\Lambda|_{\partial \mk F_{1,k_0}}; \rel \partial \mk F_{1,k_0}]) =k_0\mc H^2(\Sigma_0).
\]
However, this contradicts with \eqref{eq:less than k0} and hence Theorem \ref{thm:two spheres} is proved.
\end{proof}

\appendix

\section{Smooth approximation of piecewise-smooth surfaces}
\label{app:approximation}

In this appendix, we describe how to approximate a family of piecewise smooth spheres by smooth spheres. We only need to do this in standard $\mb S^3\subset \mb R^4$. Our approach is to mollify the characteristic functions of the family of balls bounded by the piecewise smooth family; (see \cite{Giusti-bdd-variation}*{Section 1} for the mollification procedure for a single Caccioppoli set.) We will prove that $1/2$ is a regular value of the mollified functions by showing that the directional derivatives of certain appropriately chosen vector fields have uniform negative upper bound. Hence the $1/2$-level sets are the desired smooth families by the implicit function theorem. The key ingredient is that the dihedral angles at the singular sets are bounded away from 0.

We first recall the mollifier (see \cite{Giusti-bdd-variation}*{Section 1}) and provide some basic estimates (Lemma \ref{lem:bound integral of der} and Lemma \ref{lem:isoperi}) which will be used in the mollifications.
Let $\mc K:[0,+\infty)\to [0,1]$ be a smooth cutoff function such that $\mc K'\leq 0$ and
\[
    \mc K(s)= 1 \text{ on } [0,1/3];\quad 0\leq \mc K(s)\leq 1 \text{ on } [1/3,2/3];\quad \mc K(s)= 0 \text{ on } (2/3,+\infty).
\]
Then there exists a uniform constant $\mk L_0>0$ such that 
\begin{equation}\label{eq:def of L0}
    0\leq -\mc K'<\mk L_0.
\end{equation}

For any scaling factor $r>0$, define $\mc K_r(s):= \mc K(s/r)$ and $\wti {\mc K}_r= c_r\mc K_r$, where 
\[
    c_r:= \Big(\int_{\mb S^3} \mc K_r(|x-z|)\,\mr d\mc H^3(z)\Big)^{-1}, \quad x\in\mb S^3.
\]
For any $x\in \mb S^3$ and $r>0$, denote by 
\[  B_r(x):= \{z\in\mb S^3 : |z-x|<r\},\]
where $|\cdot|$ is the standard norm of vectors in $\mb R^4$.
Clearly, there exists a uniform constant $C$ such that for all $r\in(0,1)$ and $x\in \mb S^3$,
\[
    C^{-1}r^3<\mc H^3(B_r(x))< C r^3, \quad C^{-1} r^{-3} <  c_r < Cr^{-3}.
\]

\begin{lemma}\label{lem:bound integral of der}
There exists a uniform constant $\mk L_1>0$ such that for all $x\in \mb S^3$ and $r\in(0,1)$,
\[
    - r\int_{\mb S^3}\wti {\mc K}_r'(|z-x|)\,\mr d\mc H^3(z)< \mk L_1.
\]
\end{lemma}
\begin{proof}
Applying \eqref{eq:def of L0}, we have 
\begin{align*}
 -\int_{\mb S^3}\wti {\mc K}_r'(|z-x|)\,\mr d\mc H^3(z) &\leq c_r\int_{B_r(x)}-\mc K'\Big(\frac{|z-x|}{r}\Big)\cdot r^{-1}\,\mr d\mc H^3(z)\\
 &\leq c_r\mk L_0\mc H^3(B_r(x))\cdot r^{-1}.
\end{align*}    
\end{proof}

Let $\Omega\subset \mb S^3$ be a domain with Lipschitz boundary and $\chi_\Omega$ be the characteristic function of $\Omega$. Given $r>0$, define 
\[
    u_{\Omega,r}(x):=  \int_{\Omega} \wti {\mc K}_r(|z-x|)\,\mr d\mc H^3(z)=  \int_{\mb S^3} \chi_\Omega(z)\cdot \wti {\mc K}_r(|z-x|)\,\mr d\mc H^3(z).
\]

\begin{lemma}\label{lem:isoperi}
There exists a constant $\gamma_0>0$ such that if $r\in(0,1)$, $x\in \mb S^3$, and $\Omega\subset \mb S^3$ is a domain with a Lipschitz boundary satisfying 
\[
   u_{\Omega,r}(x)=  \int_{\Omega} \wti {\mc K}_r(|z-x|)\,\mr d\mc H^3(z)=\frac{1}{2},
\]
then
\[
    r\int_{\partial \Omega} \wti {\mc K}_r(|z-x|)\,\mr d\mc H^2(z) \geq \gamma_0.
\]
\end{lemma}
\begin{proof}
Fix $x\in \mb S^3$. By the definition of $\wti{\mc K}_r$, there exists $s\in(1/3,2/3)$ such that 
\[
    1> \int_{B_{sr}(x)} \wti {\mc K}_r(|z-x|) \,\mr d\mc H^3(z)> \frac{4}{5}.
\]
It follows that 
\[
    0<\int_{\mb S^3\setminus B_{sr}(x)} \wti {\mc K}_r(|z-x|) \,\mr d\mc H^3(z)<\frac{1}{5}.
\]
Since $\mc K'\leq 0$ and $\int_{B_r(x)} \wti {\mc K}_r(|z-x|)\,\mr d\mc H^3(z)=1$, there exists $\delta>0$ such that
\[
    \wti {\mc K}_r(|z-x|) \geq \delta\cdot c_r>0 \text{ for all } z\in B_{sr}(x).
\]
By the hypothesis on $\Omega$, we obtain 
\[
    \int_{\Omega\cap B_{sr}(x)}\wti{\mc K}_r(|z-x|)\,\mr d\mc H^3(z)\geq \int_{\Omega}\wti{\mc K}_r(|z-x|)\,\mr d\mc H^3(z)- \int_{\mb S^3\setminus  B_{sr}(x)}\wti{\mc K}_r(|z-x|)\,\mr d\mc H^3(z)> \frac{1}{2}-\frac{1}{5}=\frac{3}{10}.
\]
This, together with the bound $\wti {\mc K}_r\leq c_r$, implies that
\[
    c_r \cdot \mc H^3(\Omega\cap B_{sr}(x)) \geq \frac{3}{10}.
\]
Similarly, for the complement we have
\[
    \int_{B_{sr}(x)\setminus \Omega} \wti{\mc K}_r(|z-x|)\,\mr d\mc H^3(z) \geq \int_{B_{sr}(x)} \wti{\mc K}_r(|z-x|)\,\mr d\mc H^3(z) - \int_{ \Omega} \wti{\mc K}_r(|z-x|)\,\mr d\mc H^3(z) > \frac{4}{5}-\frac{1}{2}=\frac{3}{10}.
\]
It follows that 
\[
    c_r \cdot \mc H^3(B_{sr}(x)\setminus \Omega)\geq \frac{3}{10}.
\]
Consider the rescaled metric spaces $(B_{sr}(x),r^{-2}g_{\mb S^3})$. As such spaces converge to a Euclidean ball in $\mb R^3$ as $r\to 0$, the isoperimetric inequality in $\mb R^3$ yields a uniform constant $\mk c_1$ (independent of $r\in(0,1)$ and $x\in \mb S^3$) such that 
\[
    \mc H^2(\partial \Omega\cap B_{sr}(x))\geq \mk c_1r^2.
\]
Consequently,
\[
    \int_{\partial \Omega} \wti{\mc K}_r(|z-x|)\,\mr d \mc H^2(z) \geq  \mk c_1r^2 \cdot \delta \cdot c_r\geq \mk c_1\cdot \delta\cdot C^{-1} \cdot r^{-1}.
\]
This completes the proof of Lemma \ref{lem:isoperi}.
\end{proof}

Let ${\bm a}\in\mb S^3$ and $V_{\bm a}$ be the projection to $T\mb S^3$:
\[
  V_{\bm a}(x):= \bm a - \langle x,\bm a\rangle x.
\]
We now introduce a general formula for the directional derivative of $u_{\Omega, r}$ along $V_{\bm a}$.

\begin{lemma}\label{lem:derivative of u}
Let $\Omega\subset \mb S^3$ be a domain with Lipschitz boundary. Suppose that $\bm a\in \mb S^3$, $r>0$, $\langle x_0,\bm a\rangle=0$, $x\in B_r(x_0)$. Then 
\[
    \langle  V_{\bm a}(x),\nabla u_{\Omega,r}(x)\rangle \leq -\int_{ B_r(x)\cap \partial \Omega} \wti {\mc K}_r(|z-x|) \langle V_{\bm a}(z), \mf n\rangle \,\mr d\mc H^2(z) + 6r + \frac{3}{2} \mk L_1\cdot r,
\]
where $\mf n$ is the unit outward normal vector field of $\Omega$.
\end{lemma}
\begin{proof}
    Note that for all $y,z\in \mb S^3$,
\[
    | \langle \bm a,z \rangle | = |\langle \bm a ,z-x_0\rangle | \leq |z-x_0|, \quad \text{and} \quad \frac{|\langle y,y-z\rangle|}{|z-y|}=\frac{1}{2}|y-z|.
\]
For any $z\in B_r(x)\setminus \{x\}$, we have
\begin{align*}
 \frac{\langle x-z,V_{\bm a}(x)-V_{\bm a}(z)\rangle}{|z-x|} &\leq |\langle \bm a,x \rangle| \frac{|\langle x,x-z\rangle|}{|z-x|} +|\langle \bm a,z\rangle| \frac{|\langle z,x-z\rangle|}{|z-x|}\\
 &\leq \frac{1}{2}|x-x_0|\cdot |x-z| +  \frac{1}{2}|z-x_0|\cdot |x-z|\\
 &\leq \frac{3}{2}r^2.
\end{align*}
We now compute the directional derivative of $u_{\Omega,r}$ along $V_{\bm a}(x)$:
\begin{align*}
 \langle  V(x),\nabla u_{\Omega,r}(x)\rangle &= \int_{ B_r(x)} \chi_\Omega(z) \wti{\mc K}_r'(|z-x|) \cdot \frac{\langle x-z,V_{\bm a}(x)\rangle}{|z-x|}\,\mr d\mc H^3(z)\\
 &\leq  \int_{ B_r(x)\cap \Omega} \Big( \wti{\mc K}_r'(|z-x|) \cdot \frac{\langle x-z,V_{\bm a}(z)\rangle}{|z-x|} - \frac{3}{2} \wti {\mc K}'_r(|z-x|) \cdot r^2 \Big)\,\mr d\mc H^3(z)\\
 &=\int_{ B_r(x)\cap \Omega} \Big( -\nabla_{V_{\bm a} (z)}\wti{\mc K}_r(|z-x|) - \frac{3}{2} \wti {\mc K}'_r(|z-x|) \cdot r^2 \Big)\,\mr d\mc H^3(z).
\end{align*}
Applying the divergence theorem, we obtain
\begin{align*}
&\quad \int_{ B_r(x)\cap \Omega} -\nabla_{V_{\bm a}(z)}\wti{\mc K}_r(|z-x|)\,\mr d\mc H^3(z)\\
&=\int_{B_r(x)\cap \Omega} \Big( - \operatorname{div} \big(\wti{\mc K}_r(|z-x|)\cdot  V_{\bm a}(z)\big) + \wti{\mc K}_r(|z-x|) \operatorname{div} V_{\bm a}(z) \Big)\,\mr d\mc H^3(z)\\
 &= -\int_{ B_r(x)\cap \partial \Omega} \wti {\mc K}_r(|z-x|) \langle V_{\bm a}(z), \mf n\rangle \,\mr d\mc H^2(z)+\int_{ B_r(x)\cap \Omega} -3\langle \bm a,z\rangle \wti {\mc K}_r(|z-x|) \,\mr d\mc H^3(z)\\
 &\leq -\int_{ B_r(x)\cap \partial \Omega} \wti {\mc K}_r(|z-x|) \langle V_{\bm a}(z), \mf n\rangle \,\mr d\mc H^2(z) + 6r,
\end{align*}
where we have used $\operatorname{div}_{\mb S^3} V_{\bm a}(z) = -3\langle \bm a,z\rangle$ and $|\langle \bm a,z\rangle|\leq |z-x_0|\leq 2r$. Furthermore, by Lemma \ref{lem:bound integral of der}, we deduce that 
 \begin{align*}
 \int_{ B_r(x)\cap \Omega}  - \wti {\mc K}'_r(|z-x|) \cdot r^2 \,\mr d\mc H^3(z)\leq \mk L_1\cdot r.
\end{align*}
Combining these estimates yields
\[
    \langle  V_{\bm a}(x),\nabla u_{\Omega,r}(x)\rangle \leq -\int_{ B_r(x)\cap \partial \Omega} \wti {\mc K}_r(|z-x|) \langle V_{\bm a}(z), \mf n\rangle \,\mr d\mc H^2(z) + 6r + \frac{3}{2} \mk L_1\cdot r.
\]
This finishes the proof of Lemma \ref{lem:derivative of u}.
\end{proof}

Now we are ready to prove the mollification result.
\begin{proposition}\label{prop:mollification}
Let $X$ be a cubical complex and $\{\Phi(t)\}_{t\in X}$ be a family of open balls such that their boundaries $\partial \Phi(t)$ are piecewise smoothly embedded spheres in $\mb S^3$. Suppose that the singular set $\{\mc S_t\subset\partial \Phi(t)\}_{t\in X}$ forms a smooth family of pairwise disjoint, smoothly embedded circles, and that $\{\partial \Phi(t)\setminus \mc S_t\}_{t\in X}$ varies continuously in the $C^\infty$-topology. Furthermore, assume that for each $t\in X$, within a small tubular neighborhood of $\mc S_t$, $\partial \Phi(t)$ consists of exactly two smoothly embedded surfaces meeting along $\mc S_t$ with an outer angle $\alpha_t\in (\theta, 2\pi-\theta)$ for some fixed $\theta > 0$.

Then, for any $\epsilon>0$, there exists a family $\{ \wti \Phi(t)\}_{t\in X}$ of open balls bounded by smoothly embedded spheres such that 
\[
    \ms F\big(\wti \Phi(t),\Phi(t)\big) <\epsilon.
\]
\end{proposition}

\begin{proof}
Let $\chi_t:=\chi_{\Phi(t)}$ be the characteristic function of $\Phi(t)$. Given $t\in X$ and $r>0$, define the mollified function $u_{t,r}:\mb S^3\to [0,1]$ by 
\[
    u_{t,r}(x):= \int_{\mb S^3} \chi_t(z) \wti {\mc K}_r(|z-x|)\,\mr d\mc H^3(z).
\]
Clearly, for $x\in \mb S^3$ with $\operatorname{dist}_{\mb R^4}(x,\partial \Phi(t))>r$, we have:
\[
     \text{$u_{t,r}(x)=1$ if $x\in \Phi(t)$, \quad $u_{t,r}(x)=0$ if $x\notin\Phi(t)$.}
\]

Recall that for $t\in X$, $x_t\in \mc S_t$, and sufficiently small $r>0$, the intersection $B_r(\mc S_t)\cap \partial \Phi(t)$ consists of two embedded surfaces (with outward unit normals $\mf n_1$ and $\mf n_2$) separated by $\mc S_t$. Denote by $\nu_{x_t}$ the unit vector satisfying the following conditions:
\begin{itemize}
    \item $\nu_{x_t}$ is normal to $\mc S_t$;
    \item the angle between $\nu_{x_t}$ and $\mf n_1$ is equal to the angle between $\nu_{x_t}$ and $\mf n_2$ at $x_t$;
    \item $\nu_{x_t}$ points away from $\Phi(t)$.
\end{itemize}
For $x\in \Phi(t)\setminus \mc S_t$, let $\mf n$ denote the unit outward normal vector of $\partial \Phi(t)$.

Given $x\in \mb S^3$ and a unit vector $v\in \mb S^3$, denote by $\Gamma(x,v)$ the great sphere in $\mb S^3$ that contains $x$ and is normal to $v$. By the compactness of $X$, we can choose $r_0>0$ and a constant $\alpha_0\in(0,1/2)$ such that:
\begin{enumerate}[label=(\roman*)]
\item for all $t\in X$ and $x_t\in \mc S_t$, $\partial \Phi(t)\cap B_{2r_0}(x_t)$ is a Lipschitz graph over a domain of $\Gamma(x_t,\nu_{x_t})$, and 
\begin{equation}\label{eq:alpha0 lower bound}
    \nu_{x_{t}} \cdot \mf n_i(z) > \alpha_0>0, \quad \forall i=1,2 \text{ and } z\in B_{2r_0}(x_t)\cap\partial \Phi(t),
\end{equation}
where $\mf n_i$ and $\nu_{x_t}$ are treated as vectors in $\mb R^4$;

\item for all $x\in \partial \Phi(t)\setminus B_{r_0}(\mc S_t)$, $\partial \Phi(t) \cap B_{r_0}(x)$ is a graph over a domain of $\Gamma(x,\mf n(x))$, and 
\begin{equation}\label{eq:1/2 lower bound}
   |\mf n(x)-\mf n(z)| \leq C|x-z|,\quad  \mf n(x) \cdot \mf n(z) \geq \frac{1}{2}, \quad \forall z\in B_{r_0}(x),
\end{equation}
where $\mf n(x)$ and $\mf n(z)$ are treated as vectors in $\mb R^4$, $C$ is the constant depending on the upper bound of the second fundamental forms of $\Phi(t)\setminus \mc S_t$ (which is uniformly bounded);

\item for all $t\in X, r\in(0,2r_0)$ and $x\in \partial \Phi(t)$,
\begin{equation}\label{eq:density upper bound}
    \mc H^2( B_r(x)\cap \partial \Phi(t) ) \leq C_1r^2.
\end{equation}
\end{enumerate}

\begin{claim}
There exists a constant $r_1\in (0,r_0/2)$ such that for all $r\in(0,r_1)$, $x\in \mb S^3$, and $t\in X$ satisfying $u_{t,r}(x)=1/2$, we have
\[
    \nabla u_{t,r}(x)\neq 0.
\]
\end{claim}

\begin{proof}[Proof of Claim]
Suppose that $\operatorname{dist}_{\mb R^4}(x, \partial \Phi(t))\geq r$. Then $u_{t,r}(x)$ must be either $0$ or $1$, which contradicts the assumption that $u_{t,r}(x)=1/2$. Hence, we divide the proof into two cases.

\smallskip

{\bf Case I: There exists $x_t\in \mc S_t$ such that $|x_t-x|<2r$}. Recall that
\[
    V_{\nu_{x_t}}(z):= \nu_{x_t}- \langle \nu_{x_t},z\rangle z, \quad \text{for all } z\in \mb S^3.
\]
By Lemma \ref{lem:derivative of u} 
\begin{align*}
 \langle  V_{\nu_{x_t}}(x),\nabla u_{t,r}(x)\rangle \leq -\int_{ B_r(x)\cap \partial \Phi(t)} \wti {\mc K}_r(|z-x|) \langle V_{\nu_{x_t}}(z), \mf n\rangle \,\mr d\mc H^2(z) + 6r + \frac{3}{2} \mk L_1\cdot r.
 \end{align*}
Plugging Lemma \ref{lem:isoperi} and  \eqref{eq:alpha0 lower bound} into it, we obtain
\begin{equation}\label{eq:derivative<0 in case1}
     \langle  V_{\nu_{x_t}}(x),\nabla u_{t,r}(x)\rangle \leq - \alpha_0 \gamma_0\cdot r^{-1} +(\frac{3}{2}\mk L_1+6) r<0
\end{equation}
by choosing $r_1<\sqrt{\frac{\alpha_0\gamma_0}{3\mk L_1/2+6}}$.
 
\medskip
{\bf Case II: $\operatorname{dist}_{\mb R^4}(x,\mc S_t)\geq 2r$ and there exists $x_0\in \partial \Phi(t)\setminus B_{r}(\mc S_t)$ such that $|x-x_0|<r$}.
Recall that 
\[     V_{\mf n(x_0)}(z):= \mf n(x_0) - \langle \mf n(x_0), z\rangle z, \quad \text{ for all $z\in \mb S^3$}.
\]
By Lemma \ref{lem:derivative of u} 
\begin{align*}
  \langle V_{\mf n(x_0)}(x),\nabla u_{t,r}(x)\rangle 
  \leq -\int_{ B_r(x)\cap \partial \Phi(t)} \wti {\mc K}_r(|z-x|) \langle V_{\mf n(x_0)}(z), \mf n\rangle \,\mr d\mc H^2(z) 
  + 6r + \frac{3}{2} \mk L_1\cdot r.
 \end{align*}
Plugging Lemma \ref{lem:isoperi} and  \eqref{eq:1/2 lower bound} into it, we obtain
\begin{equation}\label{eq:derivative<0 in case 2}
    \langle V_{\mf n(x_0)}(x),\nabla u_{t,r}(x)\rangle \leq -\frac{1}{2} \cdot \gamma_0 \cdot r^{-1} + 6r + \frac{3}{2}\mk L_1\cdot r < 0
\end{equation}
provided that $r_1<\sqrt{\frac{\gamma_0}{3\mk L_1+12}}$ and $r\in(0,r_1)$. This finishes the proof of Claim.
\end{proof}

Now we define 
\[
    \wti \Phi_r(t):= \{x\in \mb S^3 : u_{t,r}(x)>1/2\},\quad \forall t\in X.
\]
We note that by \eqref{eq:derivative<0 in case1} and \eqref{eq:derivative<0 in case 2},
\begin{itemize}
    \item in Case I, $\{x\in \mb S^3 : u_{t,r}(x)=1/2\}\cap B_{2r}(x_t)$ is a graph over a domain of $\Gamma(x_t,\nu_{x_t})$;
    \item in Case II, $\{x\in \mb S^3 : u_{t,r}(x)=1/2\}\cap B_{r}(x_0)$ is a graph over a domain of $\Gamma(x_0,\mf n(x_0))$. 
\end{itemize}
Hence, $\partial \wti  \Phi_r(t)$ forms a smoothly embedded sphere in $\mb S^3$. 
By the continuity of $u_{t,r}$ as $t$ varies, we conclude that $\{ \wti \Phi_r(t) \}_{t\in X}$ is a family of open balls bounded by smoothly embedded spheres. 

It remains to consider the varifold distance between $\partial \wti \Phi_r(t)$ and $\partial\Phi(t)$. Let $\bm a\in \mb S^3$ be such that $\langle \bm a ,x_0\rangle =0$ and $\langle \bm a ,\mf n(x_0)\rangle =0$. Note that by \eqref{eq:1/2 lower bound},
\[
   | \langle V_{\bm a}(z),\mf n(z)\rangle| = |\langle \bm a,\mf n(z)-\mf n(x_0)\rangle |\leq C |z-x_0|.
\]
By Lemma \ref{lem:derivative of u}, if $x\in B_r(x_0)$, we have
\begin{align*}
  \langle V_{\bm a}(x),\nabla u_{t,r}(x)\rangle 
  &\leq -\int_{ B_r(x)\cap \partial \Phi(t)} \wti {\mc K}_r(|z-x|) \langle V_{\bm a}(z), \mf n\rangle \,\mr d\mc H^2(z) 
  + 6r + \frac{3}{2} \mk L_1\cdot r\\
  &\leq 2C r \cdot c_r\cdot C_1 r^2 + 6r + \frac{3}{2} \mk L_1\cdot r,
 \end{align*}
 where we used \eqref{eq:density upper bound} in the last inequality.
Replacing $\bm a$ by $-\bm a$, we conclude that 
\[
    |\langle V_{\bm a}(x),\nabla u_{t,r}(x)\rangle | \leq C_2,
\]
where $C_2$ is a uniform constant depending on $C, C_1, \gamma_0,$ and $\mk L_1$. This together with \eqref{eq:derivative<0 in case 2} and the implicit function theorem implies that as $r\to 0$, $\partial \wti \Phi_r(t)$ converges to $\partial \Phi(t)$ in the Lipschitz sense away from $\mc S_t$. Combining with the structure of $\partial \wti \Phi_r(t)$ around $\mc S_t$, we ensure that it converges to $\partial \Phi(t)$ in the sense of varifolds. By the compactness of $X$, we can choose $r$ small enough such that 
\[
    \ms F \big(\wti \Phi_r(t),\Phi(t)\big)<\epsilon.
\]
This completes the proof of Proposition \ref{prop:mollification}.
\end{proof}

\begin{remark}\label{rmk:unifor perturb}
Recall that in the construction of $\Phi_\Lambda$ in Definition \ref{def:Phi_Lambda}, the procedure involves removing disks from $\Upsilon(x_i)$ and $\Upsilon(x_{i+1})$, and subsequently attaching a cylindrical neck to join the remaining components. Since the deformation retraction $\mc D$ remains strictly transverse to $\Upsilon(t)$, the attached neck intersects $\Upsilon(t)$ at an angle strictly bounded away from $0$ and $\pi$. Consequently, for every $x \in \Delta_k \setminus \partial \Delta_k$, the surface $\Phi_\Lambda(x)$ satisfies the required angle condition in Proposition \ref{prop:mollification}. By compactness and continuity, this allows us to extract a uniform bound $\theta > 0$ when analyzing $\Phi_\Lambda$ over the domain $\Delta_k^{\nu}$ in Section \ref{sec:smooth family}.
\end{remark}

\bibliographystyle{amsalpha}
\bibliography{minmax}
\end{document}

%% file: spheres.eps_tex
\begingroup%
  \makeatletter%
  \providecommand\color[2][]{%
    \errmessage{(Inkscape) Color is used for the text in Inkscape, but the package 'color.sty' is not loaded}%
    \renewcommand\color[2][]{}%
  }%
  \providecommand\transparent[1]{%
    \errmessage{(Inkscape) Transparency is used (non-zero) for the text in Inkscape, but the package 'transparent.sty' is not loaded}%
    \renewcommand\transparent[1]{}%
  }%
  \providecommand\rotatebox[2]{#2}%
  \newcommand*\fsize{\dimexpr\f@size pt\relax}%
  \newcommand*\lineheight[1]{\fontsize{\fsize}{#1\fsize}\selectfont}%
  \ifx\svgwidth\undefined%
    \setlength{\unitlength}{362.83464567bp}%
    \ifx\svgscale\undefined%
      \relax%
    \else%
      \setlength{\unitlength}{\unitlength * \real{\svgscale}}%
    \fi%
  \else%
    \setlength{\unitlength}{\svgwidth}%
  \fi%
  \global\let\svgwidth\undefined%
  \global\let\svgscale\undefined%
  \makeatother%
  \begin{picture}(1,0.6953125)%
    \lineheight{1}%
    \setlength\tabcolsep{0pt}%
    \put(0,0){\includegraphics[width=\unitlength]{spheres.eps}}%
    \put(0.0781606,0.16982721){\color[rgb]{0,0,0}\makebox(0,0)[lt]{\lineheight{1.25}\smash{\begin{tabular}[t]{l}$\xi_1$\end{tabular}}}}%
    \put(0.88098834,0.28168273){\color[rgb]{0,0,0}\makebox(0,0)[lt]{\lineheight{1.25}\smash{\begin{tabular}[t]{l}$\xi_2$\end{tabular}}}}%
    \put(0.05563137,0.40214894){\color[rgb]{0,0,0}\makebox(0,0)[lt]{\lineheight{1.25}\smash{\begin{tabular}[t]{l}$\xi_3$\end{tabular}}}}%
    \put(0.40652257,0.08052657){\color[rgb]{0,0,0}\makebox(0,0)[lt]{\lineheight{1.25}\smash{\begin{tabular}[t]{l}$\Upsilon(x_1)\cap \mathcal D(\xi_1)$\end{tabular}}}}%
    \put(0.33370809,0.22369665){\color[rgb]{0,0,0}\makebox(0,0)[lt]{\lineheight{1.25}\smash{\begin{tabular}[t]{l}$\Upsilon(x_2)\cap \mathcal D(\xi_1)\setminus \mathcal D(1-\xi_2)$\end{tabular}}}}%
    \put(0.31639142,0.32929913){\color[rgb]{0,0,0}\makebox(0,0)[lt]{\lineheight{1.25}\smash{\begin{tabular}[t]{l}$\Upsilon(x_3)\cap \mathcal D(\xi_3)\setminus \mathcal D(1-\xi_2)$\end{tabular}}}}%
    \put(0.28628176,0.56124778){\color[rgb]{0,0,0}\makebox(0,0)[lt]{\lineheight{1.25}\smash{\begin{tabular}[t]{l}$\partial \mathcal{D}(\xi_3)\cap (\Omega_{x_{4}}\setminus\oli{\Omega_{x_3}})$\end{tabular}}}}%
    \put(0.46246325,0.44313357){\color[rgb]{0,0,0}\makebox(0,0)[lt]{\lineheight{1.25}\smash{\begin{tabular}[t]{l}$\partial \mathcal{D}(1-\xi_2)\cap (\Omega_{x_{3}}\setminus\oli{\Omega_{x_2}})$\end{tabular}}}}%
  \end{picture}%
\endgroup%

%% file: two_minimal_spheres.bbl
\begin{bibdiv}
\begin{biblist}

\bib{Alm62}{article}{
      author={Almgren, Frederick~Justin, Jr.},
       title={The homotopy groups of the integral cycle groups},
        date={1962},
        ISSN={0040-9383},
     journal={Topology},
      volume={1},
       pages={257\ndash 299},
         url={http://dx.doi.org/10.1016/0040-9383(62)90016-2},
      review={\MR{0146835}},
}

\bib{Alm65}{article}{
      author={Almgren, Frederick~Justin, Jr.},
       title={The theory of varifolds},
        date={1965},
     journal={Mimeographed notes},
}

\bib{AMN-rigidity}{article}{
      author={Ambrozio, Lucas},
      author={Marques, Fernando~C.},
      author={Neves, André},
       title={Rigidity theorems for the area widths of {R}iemannian manifolds},
        date={2024},
      eprint={2408.14375},
         url={https://arxiv.org/abs/2408.14375},
}

\bib{Bam-Klei-smale}{article}{
      author={Bamler, Richard~H.},
      author={Kleiner, Bruce},
       title={Ricci flow and contractibility of spaces of metrics},
        date={2019},
      eprint={1909.08710},
         url={https://arxiv.org/abs/1909.08710},
}

\bib{CS}{article}{
      author={Choi, Hyeong~In},
      author={Schoen, Richard},
       title={The space of minimal embeddings of a surface into a
  three-dimensional manifold of positive {R}icci curvature},
        date={1985},
        ISSN={0020-9910,1432-1297},
     journal={Invent. Math.},
      volume={81},
      number={3},
       pages={387\ndash 394},
         url={https://doi.org/10.1007/BF01388577},
      review={\MR{807063}},
}

\bib{chu-all-genus}{article}{
      author={Chu, Adrian Chun-Pong},
       title={Minimal surfaces with arbitrary genus in 3-spheres of positive
  ricci curvature},
        date={2025},
      eprint={2508.06019},
         url={https://arxiv.org/abs/2508.06019},
}

\bib{chu-Li-tori}{article}{
      author={Chu, Adrian Chun-Pong},
      author={Li, Yangyang},
       title={Existence of 5 minimal tori in 3-spheres of positive ricci
  curvature},
        date={2025},
      eprint={2409.09315},
         url={https://arxiv.org/abs/2409.09315},
}

\bib{Chu-Li-Wang-enumerative}{article}{
      author={Chu, Adrian Chun-Pong},
      author={Li, Yangyang},
      author={Wang, Zhihan},
       title={An enumerative min-max theorem for minimal surfaces},
        date={2026},
      eprint={2601.01736},
         url={https://arxiv.org/abs/2601.01736},
}

\bib{Chu-Li-Wang-genus}{article}{
      author={Chu, Adrian Chun-Pong},
      author={Li, Yangyang},
      author={Wang, Zhihan},
       title={Min-max theory and minimal surfaces with prescribed genus},
        date={2026},
      eprint={2507.23239},
         url={https://arxiv.org/abs/2507.23239},
}

\bib{Colding-DeLellis03}{incollection}{
      author={Colding, Tobias~H.},
      author={De~Lellis, Camillo},
       title={The min-max construction of minimal surfaces},
        date={2003},
   booktitle={Surveys in differential geometry, {V}ol. {VIII} ({B}oston, {MA},
  2002)},
      series={Surv. Differ. Geom.},
      volume={8},
   publisher={Int. Press, Somerville, MA},
       pages={75\ndash 107},
         url={https://doi.org/10.4310/SDG.2003.v8.n1.a3},
      review={\MR{2039986}},
}

\bib{Colding-Minicozzi08b}{article}{
      author={Colding, Tobias~H.},
      author={Minicozzi, William~P., II},
       title={Width and finite extinction time of {R}icci flow},
        date={2008},
        ISSN={1465-3060},
     journal={Geom. Topol.},
      volume={12},
      number={5},
       pages={2537\ndash 2586},
         url={https://doi.org/10.2140/gt.2008.12.2537},
      review={\MR{2460871}},
}

\bib{Deaibes2-foliations}{article}{
      author={Deaibes, Salim},
       title={Minimal $2$-spheres and optimal foliations in $3$-spheres with
  arbitrary metric},
        date={2021},
      eprint={2001.07695},
         url={https://arxiv.org/abs/2001.07695},
}

\bib{Fed}{book}{
      author={Federer, Herbert},
       title={Geometric measure theory},
      series={Die Grundlehren der mathematischen Wissenschaften, Band 153},
   publisher={Springer-Verlag New York Inc., New York},
        date={1969},
      review={\MR{0257325}},
}

\bib{BP-nonplanar}{article}{
      author={Ghini~Bettiol, Renato},
      author={Piccione, Paolo},
       title={Nonplanar minimal spheres in ellipsoids of revolution},
        date={2024-05},
        ISSN={0391-173X},
     journal={Ann. Sc. Norm. Super. Pisa Cl. Sci},
       pages={32},
         url={http://dx.doi.org/10.2422/2036-2145.202309_018},
}

\bib{Giusti-bdd-variation}{book}{
      author={Giusti, Enrico},
       title={Minimal surfaces and functions of bounded variation},
      series={Monographs in Mathematics},
   publisher={Birkh\"auser Verlag, Basel},
        date={1984},
      volume={80},
        ISBN={0-8176-3153-4},
         url={https://doi.org/10.1007/978-1-4684-9486-0},
      review={\MR{775682}},
}

\bib{Gro88}{incollection}{
      author={Gromov, Mikhael},
       title={Dimension, nonlinear spectra and width},
        date={1988},
   booktitle={Geometric aspects of functional analysis (1986/87)},
      series={Lecture Notes in Math.},
      volume={1317},
   publisher={Springer, Berlin},
       pages={132\ndash 184},
         url={https://doi.org/10.1007/BFb0081739},
      review={\MR{950979}},
}

\bib{HK19}{article}{
      author={Haslhofer, Robert},
      author={Ketover, Daniel},
       title={Minimal 2-spheres in 3-spheres},
        date={2019},
        ISSN={0012-7094},
     journal={Duke Math. J.},
      volume={168},
      number={10},
       pages={1929\ndash 1975},
         url={https://doi.org/10.1215/00127094-2019-0009},
      review={\MR{3983295}},
}

\bib{Hat83}{article}{
      author={Hatcher, Allen~E.},
       title={A proof of the {S}male conjecture, {${\rm Diff}(S^{3})\simeq {\rm
  O}(4)$}},
        date={1983},
        ISSN={0003-486X},
     journal={Ann. of Math. (2)},
      volume={117},
      number={3},
       pages={553\ndash 607},
         url={https://doi.org/10.2307/2007035},
      review={\MR{701256}},
}

\bib{Hat-AT-book}{book}{
      author={Hatcher, Allen~E.},
       title={Algebraic topology},
   publisher={Cambridge University Press, Cambridge},
        date={2002},
        ISBN={0-521-79160-X; 0-521-79540-0},
      review={\MR{1867354}},
}

\bib{IMN17}{article}{
      author={Irie, Kei},
      author={Marques, Fernando~C.},
      author={Neves, Andr\'e},
       title={Density of minimal hypersurfaces for generic metrics},
        date={2018},
        ISSN={0003-486X},
     journal={Ann. of Math. (2)},
      volume={187},
      number={3},
       pages={963\ndash 972},
         url={https://doi.org/10.4007/annals.2018.187.3.8},
      review={\MR{3779962}},
}

\bib{Liokomovich-Ketover23}{article}{
      author={Ketover, Daniel},
      author={Liokumovich, Yevgeny},
       title={The {S}male conjecture and min-max theory},
        date={2025},
        ISSN={0020-9910,1432-1297},
     journal={Invent. Math.},
      volume={241},
      number={1},
       pages={1\ndash 34},
         url={https://doi.org/10.1007/s00222-025-01334-z},
      review={\MR{4921571}},
}

\bib{KMN16}{article}{
      author={Ketover, Daniel},
      author={Marques, Fernando~C.},
      author={Neves, Andr\'e},
       title={The catenoid estimate and its geometric applications},
        date={2020},
        ISSN={0022-040X},
     journal={J. Differential Geom.},
      volume={115},
      number={1},
       pages={1\ndash 26},
         url={https://doi.org/10.4310/jdg/1586224840},
      review={\MR{4081930}},
}

\bib{Li-W-X-lens}{article}{
      author={Li, Xingzhe},
      author={Wang, Tongrui},
      author={Yao, Xuan},
       title={Minimal surfaces with low genus in lens spaces},
        date={2025},
        ISSN={0075-4102,1435-5345},
     journal={J. Reine Angew. Math.},
      volume={828},
       pages={175\ndash 218},
         url={https://doi.org/10.1515/crelle-2025-0061},
      review={\MR{4979239}},
}

\bib{li-Wang-tori}{article}{
      author={Li, Xingzhe},
      author={Wang, Zhichao},
       title={Existence of embedded minimal tori in three-spheres with positive
  {R}icci curvature},
        date={2024},
      eprint={2409.10391},
         url={https://arxiv.org/abs/2409.10391},
}

\bib{LMN16}{article}{
      author={Liokumovich, Yevgeny},
      author={Marques, Fernando~C.},
      author={Neves, Andr\'e},
       title={Weyl law for the volume spectrum},
        date={2018},
        ISSN={0003-486X},
     journal={Ann. of Math. (2)},
      volume={187},
      number={3},
       pages={933\ndash 961},
         url={https://doi.org/10.4007/annals.2018.187.3.7},
      review={\MR{3779961}},
}

\bib{MN14}{article}{
      author={Marques, {Fernando C.}},
      author={Neves, Andr\'e},
       title={Min-max theory and the {W}illmore conjecture},
        date={2014},
        ISSN={0003-486X},
     journal={Ann. of Math. (2)},
      volume={179},
      number={2},
       pages={683\ndash 782},
         url={http://dx.doi.org/10.4007/annals.2014.179.2.6},
      review={\MR{3152944}},
}

\bib{MN17}{article}{
      author={Marques, {Fernando C.}},
      author={Neves, Andr\'e},
       title={Existence of infinitely many minimal hypersurfaces in positive
  {R}icci curvature},
        date={2017},
        ISSN={0020-9910},
     journal={Invent. Math.},
      volume={209},
      number={2},
       pages={577\ndash 616},
         url={http://dx.doi.org/10.1007/s00222-017-0716-6},
      review={\MR{3674223}},
}

\bib{MN18}{article}{
      author={Marques, {Fernando C.}},
      author={Neves, Andr\'e},
       title={Morse index of multiplicity one min-max minimal hypersurfaces},
        date={2021},
        ISSN={0001-8708},
     journal={Adv. Math.},
      volume={378},
       pages={Paper No. 107527, 58},
         url={https://doi.org/10.1016/j.aim.2020.107527},
      review={\MR{4191255}},
}

\bib{MNS17}{article}{
      author={Marques, {Fernando C.}},
      author={Neves, Andr\'e},
      author={Song, Antoine},
       title={Equidistribution of minimal hypersurfaces for generic metrics},
        date={2019},
        ISSN={0020-9910},
     journal={Invent. Math.},
      volume={216},
      number={2},
       pages={421\ndash 443},
         url={https://doi.org/10.1007/s00222-018-00850-5},
      review={\MR{3953507}},
}

\bib{Pi}{book}{
      author={Pitts, Jon~T.},
       title={Existence and regularity of minimal surfaces on {R}iemannian
  manifolds},
      series={Mathematical Notes},
   publisher={Princeton University Press, Princeton, N.J.; University of Tokyo
  Press, Tokyo},
        date={1981},
      volume={27},
        ISBN={0-691-08290-1},
      review={\MR{626027}},
}

\bib{Ritore2023}{book}{
      author={Ritor{\'e}, Manuel},
       title={Isoperimetric inequalities in riemannian manifolds},
      series={Progress in Mathematics},
   publisher={Birkh{\"a}user/Springer},
     address={Cham},
        date={2023},
      volume={348},
        ISBN={978-3-031-37900-0},
      review={\MR{4676392}},
}

\bib{Sacks-Uhlenbeck81}{article}{
      author={Sacks, Jonathan},
      author={Uhlenbeck, Karen~K.},
       title={The existence of minimal immersions of {$2$}-spheres},
        date={1981},
        ISSN={0003-486X},
     journal={Ann. of Math. (2)},
      volume={113},
      number={1},
       pages={1\ndash 24},
         url={https://doi.org/10.2307/1971131},
      review={\MR{604040}},
}

\bib{SS23}{article}{
      author={Sarnataro, Lorenzo},
      author={Stryker, Douglas},
       title={Optimal regularity for minimizers of the prescribed mean
  curvature functional over isotopies},
        date={2025},
        ISSN={2168-0930,2168-0949},
     journal={Camb. J. Math.},
      volume={13},
      number={3},
       pages={609\ndash 706},
         url={https://doi.org/10.4310/cjm.250715011408},
      review={\MR{4933237}},
}

\bib{SS}{article}{
      author={Schoen, Richard},
      author={Simon, Leon},
       title={Regularity of stable minimal hypersurfaces},
        date={1981},
        ISSN={0010-3640},
     journal={Comm. Pure Appl. Math.},
      volume={34},
      number={6},
       pages={741\ndash 797},
         url={http://dx.doi.org/10.1002/cpa.3160340603},
      review={\MR{634285 (82k:49054)}},
}

\bib{Si}{book}{
      author={Simon, Leon},
       title={Lectures on geometric measure theory},
      series={Proceedings of the Centre for Mathematical Analysis, Australian
  National University},
   publisher={Australian National University, Centre for Mathematical Analysis,
  Canberra},
        date={1983},
      volume={3},
        ISBN={0-86784-429-9},
      review={\MR{756417 (87a:49001)}},
}

\bib{Smith82}{thesis}{
      author={Smith, Francis~R.},
       title={On the existence of embedded minimal 2-spheres in the 3-sphere,
  endowed with an arbitrary {R}iemannian metric},
        type={Ph.D. Thesis},
        date={1982},
}

\bib{Song18}{article}{
      author={Song, Antoine},
       title={Existence of infinitely many minimal hypersurfaces in closed
  manifolds},
        date={2023},
        ISSN={0003-486X},
     journal={Ann. of Math. (2)},
      volume={197},
      number={3},
       pages={859\ndash 895},
         url={https://doi.org/10.4007/annals.2023.197.3.1},
      review={\MR{4564260}},
}

\bib{WWZ-S4}{article}{
      author={Wang, Tongrui},
      author={Wang, Zhichao},
      author={Zhou, Xin},
       title={Equivariant min-max theory and the spherical bernstein problem in
  $\mathbb{S}^4$},
        date={2026},
      eprint={2602.03984},
         url={https://arxiv.org/abs/2602.03984},
}

\bib{Wang-Zhou-C11-2023}{article}{
      author={Wang, Zhichao},
      author={Zhou, Xin},
       title={Improved ${C}^{1,1}$ regularity for multiple membranes problem},
        date={2023},
      eprint={2308.00172},
}

\bib{wangzc2023existenceFour}{article}{
      author={Wang, Zhichao},
      author={Zhou, Xin},
       title={Existence of four minimal spheres in ${S}^3$ with a bumpy
  metric},
        date={2024},
      eprint={2305.08755},
         url={https://arxiv.org/abs/2305.08755},
}

\bib{Wang-Zhou22}{article}{
      author={Wang, Zhichao},
      author={Zhou, Xin},
       title={Min-max minimal hypersurfaces with higher multiplicity},
        date={2025},
        ISSN={0022-040X,1945-743X},
     journal={J. Differential Geom.},
      volume={129},
      number={1},
       pages={225\ndash 260},
         url={https://doi.org/10.4310/jdg/1736262245},
      review={\MR{4856595}},
}

\bib{Whi87}{article}{
      author={White, Brian},
       title={Curvature estimates and compactness theorems in {$3$}-manifolds
  for surfaces that are stationary for parametric elliptic functionals},
        date={1987},
        ISSN={0020-9910},
     journal={Invent. Math.},
      volume={88},
      number={2},
       pages={243\ndash 256},
         url={http://dx.doi.org/10.1007/BF01388908},
      review={\MR{880951}},
}

\bib{Whi91}{article}{
      author={White, Brian},
       title={The space of minimal submanifolds for varying {R}iemannian
  metrics},
        date={1991},
        ISSN={0022-2518},
     journal={Indiana Univ. Math. J.},
      volume={40},
      number={1},
       pages={161\ndash 200},
         url={http://dx.doi.org/10.1512/iumj.1991.40.40008},
      review={\MR{1101226 (92i:58028)}},
}

\bib{Yau82}{incollection}{
      author={Yau, Shing-Tung},
       title={Problem section},
        date={1982},
   booktitle={Seminar on {D}ifferential {G}eometry},
      series={Ann. of Math. Stud.},
      volume={102},
   publisher={Princeton Univ. Press, Princeton, N.J.},
       pages={669\ndash 706},
      review={\MR{645762}},
}

\bib{Zhou19}{article}{
      author={Zhou, Xin},
       title={On the {M}ultiplicity {O}ne {C}onjecture in min-max theory},
        date={2020},
        ISSN={0003-486X},
     journal={Ann. of Math. (2)},
      volume={192},
      number={3},
       pages={767\ndash 820},
         url={https://doi.org/10.4007/annals.2020.192.3.3},
      review={\MR{4172621}},
}

\end{biblist}
\end{bibdiv}
